\documentclass[11pt]{article}
\usepackage[ansinew]{inputenc}

\usepackage{latexsym,bezier,enumerate,longtable,dcolumn,color,xspace,curves,mathpazo,url, pstricks}
\usepackage{amsmath,amsthm,amsopn,amstext,amscd,amsfonts,amssymb,mathtools,fancybox,pst-node,pst-tree}

\usepackage{epsfig,epic,graphicx,pstricks,epsfig,graphics}
\usepackage{pgf,tikz}

\usepackage{booktabs,natbib,comment}
\usepackage{enumitem}
\newcommand{\bd}[1]{}

\usepackage{multirow}
\usepackage{diagbox}
\usepackage{array}

\usepackage{rotating}
\usepackage{colortbl}

\usepackage{siunitx}
\mathchardef\isinpunto="0010

\newcommand{\G}{{\Gamma}}

\newcommand{\SR}{\mathbb{R}}
\DeclareMathSymbol{\shortminus}{\mathbin}{AMSa}{"39}
\newcommand{\NEG}[1]{\shortminus#1}

\newcommand{\cE}{{\cal E}}
\newcommand{\cG}{{\cal G}}
\newcommand{\cI}{{\cal I}}
\newcommand{\cK}{{\cal K}}
\newcommand{\cM}{{\cal M}}
\newcommand{\cN}{{\cal N}}

\newcommand{\cU}{{\cal U}}
\newcommand{\cC}{{\cal C}}

\theoremstyle{plain}
\newtheorem{proposition}{\bf Proposition}
\newtheorem{theorem}{\bf Theorem}

\theoremstyle{definition}
\newtheorem{definition}{\bf Definition}
\newtheorem{example}{\bf Example}

\title{\baselineskip .2in
\bf Myerson Interaction Index\footnote{This research has been
supported by I+D+i research projects  PID2020-116884GB-I00 from the Government of Spain, and PR27/25-32473 from Universidad Complutense de Madrid. }}
\date{\today}
\author{}

\begin{document}
\setlength{\abovedisplayskip}{5mm} 
\setlength{\abovedisplayshortskip}{3mm} 
\setlength{\belowdisplayskip}{6mm} 
\setlength{\belowdisplayshortskip}{5mm} 

\maketitle \vspace*{-1cm} {\baselineskip .2in}
\begin{center}
{\small
\begin{tabular}{l} \bf Jorge Gonz\'alez-Ortega$^1$
\\ e-mail: \url{jgortega@ucm.es}
\\
\bf Elisenda Molina$^1$
\\ e-mail: \url{elisenda.molina@ucm.es}
\\
\bf Juan Tejada$^1$
\\  e-mail: \url{jtejada@mat.ucm.es}
\\[2mm] $ ^1$ Instituto de Matem\'atica Interdisciplinar (IMI), Dpt. Estad\'{\i}stica e Investigaci\'on Operativa,
\\ Universidad Complutense de Madrid, Spain
\end{tabular}}
\end{center}
%


\begin{abstract}
This paper introduces the Myerson interaction index (MII), an extension of the Shapley interaction index to cooperative games with communication structures restricted by graphs. We establish a formal framework for interaction indices on graphs and provide an axiomatic characterization of the MII based on component efficiency, fairness, and veto partnership consistency. Furthermore, we analyse network-induced interaction to distinguish between effects arising from the graph topology and those stemming from the game's intrinsic synergies. This framework offers a robust tool for analysing coalitional behaviour in structured environments, with potential applications in social network analysis and explainable artificial intelligence.

\end{abstract}


\section{Introduction}
\label{sec:intro}

The concept of interaction lies at the heart of numerous disciplines, serving as a key mechanism to explain how entities - be they players, variables, or features - jointly influence outcomes in a system. For instance, in game theory, regression analysis, and machine learning, the study of interaction has evolved from basic pairwise effects to sophisticated representations of synergy, redundancy, and cooperation. A unifying perspective emerges from cooperative game theory, where 
 interaction is naturally associated with the additional value generated by cooperation. Early contributions by \cite{Ow72} introduced the idea of measuring how the presence of one player affects the marginal contribution of another. This idea was later significantly developed by \cite{GrRo99}, who proposed a general axiomatic framework for interaction indices based on the Shapley and Banzhaf values. Their approach provides a measure of the interaction effects of arbitrary order, providing a rich and interpretable representation of cooperative behaviour.

Beyond pure game theoretical interest, those interaction indices have found important applications in other areas. In multicriteria decision making \cite{Gr96} used them to model interactions between criteria, overcoming the limitations of additive aggregation models. Closely related ideas appear in the theory of fuzzy measures and capacities, where interaction indices provide a quantitative description of synergy and redundancy among attributes, with notable contributions including those by  \cite{Gr97a} and \cite{GrLa2016}, as well as applications discussed by \cite{CaInML2025}. In these contexts, interaction indices serve as a fundamental tool for representing complex preference structures and non-additive phenomena.

More recently, interaction measures have gained renewed relevance in the context of machine learning and, in particular, in the field of eXplainable Artificial Intelligence (XAI). Modern predictive black-models often capture complex, high-order interactions between features, but these interactions are typically implicit and they are not directly measured. Game-theoretic explanation methods, such as SHAP values, rely on the Shapley value and related interaction indices to provide model-agnostic explanations of both individual feature contributions and their interactions. This has led to the development of several Shapley-based interaction measures tailored to XAI (\cite{SuDhAg20} and \cite{TsYeRa23}), further highlighting the versatility and relevance of the cooperative game-theoretic framework.

In many real-world situations, however, cooperation or interaction is not unrestricted. Agents may only interact or coordinate through an underlying communication or interaction structure, which is naturally represented by a graph. In such settings, ignoring the network constraints may lead to a distorted assessment of both individual contributions and interaction effects. Myerson \cite{My77} addressed this issue by introducing graph-restricted games and the Myerson value, which adapts the Shapley value to situations with limited communication.

The aim of this paper is to extend the concept of interaction indices to cooperative games with restricted communication structures. Building on the Shapley Interaction Index and the Myerson restricted game, we introduce a graph-based interaction index that captures both the intrinsic complementarities of the underlying game and the constraints imposed by the network. We provide an axiomatic characterization of the proposed index and illustrate how the communication structure reshapes interaction patterns, giving rise to what we interpret as \emph{network -- induced interaction}. 

The Myerson interaction index provides a framework for studying interactions among agents, attributes, or features connected through relationships such as influence, compatibility, dependency, or cooperation. In particular, it can be applied to relevant problems in social networks, including link formation and community detection. In this setting, the proposed framework complements existing approaches based on interaction indices for games constructed from graphs, such as \cite{SzBaMiRa2015} and \cite{TaMiWool2019}. Rather than deriving the game itself from the graph structure, our approach studies interactions in cooperative games where feasible cooperation is restricted by a graph. In this way, it captures how relational constraints shape cooperative opportunities. Moreover, the index may also contribute to explainability in Artificial Intelligence when data are enriched with graph structures where features are represented as nodes and edges capture semantic, functional, or causal relationships. This enables richer analyses than traditional approaches based solely on independent features.

This article is organized as follows. Section \ref{sec:prelim} introduces the preliminary concepts from game theory and graph theory required for the remainder of the paper. Section \ref{sec:MII} presents the definition of a \emph{graph interaction index} and introduces the Myerson interaction index, together with its main properties and axiomatic characterization. Section \ref{sec:applic} is devoted to the notion of \emph{network-induced interaction}, which isolates the effect of the communication structure on interactions from the purely game-theoretic component. Finally, the last section concludes the paper.


\section{Preliminaries}
\label{sec:prelim}

In this section, the basic concepts about coalitional games, interaction indices and graphs are presented.

\subsection{Coalitional games}
\label{subsec:prelim_TU}

A \emph{coalitional game} with transferable utility (TU game) is a pair $(N,v)$ where $N = \{1,\dots,n\}$ is the set of players and $v$, the \emph{characteristic function}, is a map $v: 2^N \rightarrow \SR$, with $v(\emptyset) = 0$. For each coalition, $S \subseteq N$, $v(S)$ represents the maximum transferable utility that $S$ can be guaranteed to obtain whenever its members cooperate. For brevity, throughout the paper, the cardinality of sets (coalitions) $N$, $S$ and $C$ will be denoted by appropriate small letters $n$, $s$ and $c$, respectively. Also, for notational convenience, we will write singleton $\{i\}$ as $i$, when no ambiguity appears.

A game $(N,v)$ is \emph{superadditive} if $v(S \cup T) \geq v(S) + v(T)$, $\forall\, S,T \subseteq N$ with $S \cap T = \emptyset$, and it is \emph{convex} if $v(S \cup T) \geq v(S) + v(T) - v(S \cap T)$, $\forall\, S,T \subseteq N$.

Given a game $(N,v)$ and a subset $C \subseteq N$, $C \neq \emptyset$, the \emph{quotient game} with respect to $C$ is defined as the $(n-c+1)$-person game $\left([C] \cup (N \setminus C),v_{[C]}\right)$, where coalition $C$ is replaced by the single player $[C]$ which acts as a proxy $[C] \equiv C$, and $v_{[C]}$ is of the form
\begin{equation*}
  v_{[C]}(S) = v(S) \text{ and } v_{[C]}(S\cup [C]) = v(S\cup C), 
  \quad \forall\, S \subseteq N \setminus C.
\end{equation*}

Let $G^N$ be the class of all coalitional games $(N,v)$ with player set $N$. For notational convenience, we will write games $(N,v)$ as $v$, when no ambiguity arises. For each $T \in 2^N \setminus \{\emptyset\}$, the \emph{unanimity game} $u_T \in G^N$ is defined by
\begin{equation*}
  u_T(S) =
  \begin{cases}
    1, & \text{if } T \subseteq S \\ 
    0, & \text{otherwise}
  \end{cases},
  \quad \forall\, S \subseteq N.
\end{equation*}
The collection $\left\{u_T \,\vert\, T \in 2^N \setminus \{\emptyset\}\right\}$ is a basis of $G^N$ and for each $v \in G^N$ we have that
\begin{equation*}
  v = \sum_{T \in 2^N \setminus \{\emptyset\}} {\Delta_v(T)\,u_T}\text{,}
\end{equation*}
where $\left\{\Delta_v(T)\right\}_{T \in 2^N \setminus \{\emptyset\}}$ is the set of the unanimity coefficients of $v$ or the \emph{Harsanyi dividends} of $v$ \citep{Ha59}, which are given by
\begin{equation*}
  \Delta_v(T) = \sum_{S \subseteq T} {(\NEG{1})^{t-s}\,v(S)},
\end{equation*}
being $t = \vert T \vert$ and $s = \vert S \vert$.

As is well known, the \emph{Shapley value} of a game $v \in G^N$ \citep{Sh53} is an allocation rule $\phi: G^N \rightarrow \SR^N$ defined by
\begin{equation*}
  \phi_i(v) \coloneq
  \sum_{S \subseteq N:\, i \in S} {\frac{(s-1)!\,(n-s)!}{n!}\left(v(S) - v(S \setminus i)\right)}.
\end{equation*}
An alternative expression for this value in terms of the dividends is
\begin{equation*}
  \phi_i(v) =
  \sum_{S \subseteq N:\, i \in S} {\frac{\Delta_v(S)}{s}}.
\end{equation*}


\subsection{Shapley Interaction Index}
\label{subsec:prelim_SII}

An \emph{interaction index} $I^v$ of the game $v \in G^N$ is a function $I^v: 2^N \setminus \{\emptyset\} \rightarrow \SR$ which tries to reflect the degree of cooperation or interaction existing among players. In particular, a positive interaction value must show the interest of players to cooperate.

\cite{GrRo99} proposed an interaction index based on the Shapley value, called the \emph{Shapley Interaction Index} (SII). Given a game $v \in G^N$ and a player subset $S \subseteq N$, $S \neq \emptyset$, the SII is defined as
\begin{equation} \label{eq:SII_def_deriv}
  SI^v(S) \coloneq
  \sum_{T \subseteq N\setminus S} \frac{(n-t-s)!t!}{(n-s+1)!} \delta_S^v(T),
\end{equation}
where $\delta_S^v(T)$ denotes the \emph{$S$-derivative} of $v$ at $T$, and it is defined as
\begin{equation} \label{eq:S-deriv_def} 
 \delta_S^v(T) \coloneq \sum_{L \subseteq S} (-1)^{s-l} v(T\cup L),
 \quad \forall \, S\subseteq N,\; \forall\, T\subseteq N\setminus S.  
\end{equation}

\cite{GrRo99} gave alternative expressions for the SII. It can also be recursively determined as
\begin{equation*}
  SI^v(S) =
  SI^{v_{[S]}}([S]) - \sum_{K \subsetneq S,\, K \neq \emptyset} {SI^{v_{N \setminus K}}(S \setminus K)},
\end{equation*}
where $v_{[S]}$ is the quotient game with respect to $S$, $v_{N \setminus K}$ the restricted game to $N \setminus K$, and $SI^v(i) = \phi_i(v)$.

The SII can also be expressed in terms of Shapley group values \citep{FlMoTe19} as
\begin{equation*}
  SI^v(S) =
  \sum_{K \subseteq S,\, K \neq \emptyset} {(\NEG{1})^{s-k}\,\phi_{[K]}(v_{(N \setminus S) \cup [K]})},
\end{equation*}
where $\phi_{[K]}(v_{(N \setminus S) \cup [K]})$ is the Shapley value of the proxy player $[K]$ in the quotient game $v_{(N \setminus S) \cup [K]}$, which is precisely the \emph{Shapley group value} of coalition $K$. 

Finally, in terms of the Harsanyi dividends, the SII may be computed as
\begin{equation} \label{eq:SII_def_divid}
  SI^v(S) =
  \sum_{T \in 2^N \setminus \{\emptyset\}} {\Delta_v(T)\,SI^{u_T}(S)} =
  \sum_{T \supseteq S} {\frac{\Delta_v(T)}{t-s+1}},
\end{equation}
where the SII for unanimity games \citep{Gr97a} is
\begin{equation} \label{eq:SII_def_unanim}
  SI^{u_T}(S) =
  \begin{cases}
    \frac{1}{t-s+1}, & \text{if } S \subseteq T \\
    0, & \text{otherwise}
  \end{cases},
  \quad \forall\, T \in 2^N \setminus \{\emptyset\}.
\end{equation}


\subsection{Graphs}
\label{subsec:prelim_graphs}

Let $\G = (N,E)$ be an undirected graph, where $N$ is the set of $n$ nodes and $E$ is the set of edges, that is, $E \subseteq \left\{\{i,j\}:\, i,j \in N, i \neq j\right\}$. In addition, let $\cG^N$ denote the class of all undirected graphs with node set $N$, and $\cG^{\G}$ the collection of all subgraphs of $\G$. The \emph{subgraph of $\G$ induced by $S \subseteq N$} is $\G_S = (S,E[S]) \in \cG^{\G}$ where $E[S] = \left\{\{i,j\} \in E:\, i,j \in S\right\}$.

A {\em path} $\mu[i_1,i_r]$ from node $i_1$ to node $i_r$ is a sequence of different nodes $(i_1,\ldots,i_r)$, $r\geq 2$, and edges $E[\mu[i_1,i_r]]=(e_1,e_2,\dots,e_{r-1})$, s.t. $e_h=\{i_h,i_{h+1}\}\in E$ , for all $h=1,\ldots,r-1$.  

A graph is \emph{connected} if every pair of its nodes $i,j \in N$ they are connected directly or indirectly, i.e. if there is a path in the graph from node $i$ to node $j$; otherwise, the graph is \emph{not connected}. The relation of connectivity induces a partition of the node set $N$ into \emph{connected components}, with two nodes being in the same connected component if and only if they are connected. Let $\cK_{\G}$ denote the collection of connected components of graph $\G$. Moreover, for the sake of simplicity, $\cK_{\G}(S)$ shall designate the collection of connected components of graph $\G_S$. For every node $i\in N$, $\cK_{\G}(i)\in \cK_{\G}$ denotes the connected component of the graph $\G$ to which node $i$ belongs. 

Given graph $\G = (N,E)$ and $S \subseteq N$, $S \neq \emptyset$, a subgraph $\G_R = (R,E[R]) \in \cG^{\G}$ is called a \emph{minimal $S$-connecting subgraph} if it connects $S$ and any $\G_{R'} = (R',E[R'])$ with $R' \subsetneq R$ does not connect $S$. Let $\cM_{\G}(S)$ denote the collection of minimal $S$-connecting subgraphs, which will be described by the sets of nodes of those subgraphs, i.e. $\cM_{\G}(S) = \{S_1,\dots,S_{\ell}\}$. 

Furthermore, the set of \emph{intermediaries} of $S$ in $\G$, which will be denoted by $B_{\G}(S)$, is defined as  
\begin{equation*}
  B_{\G}(S) \coloneq \left(\bigcup_{m = 1}^{\ell} {S_m}\right) \setminus S;
\end{equation*}
while the set of \emph{essential intermediaries} of $S$ in $\G$, which will be denoted by $EB_{\G}(S)$, is defined as
\begin{equation*}
  EB_{\Gamma}(S) \coloneq \left(\bigcap_{m = 1}^{\ell} {S_m}\right) \setminus S.
\end{equation*}
Note that if $\G$ is a tree then $\cM_{\G}(S)$ has a unique element $(H_{\G}(S),E[H_{\G}(S)])$. In the sequel, $H_{\G}(S)$ will be referred to as the \emph{convex hull} of $S$ in $\G$.

To conclude, given a graph $\G = (N,E)$ and a subset $C \subseteq N$, $C \neq \emptyset$, the \emph{quotient graph} with respect to $C$ is defined as the graph $\G_{[C]}$ where subset $C$ is replaced by the single node $[C]$, which is adjacent to any other vertex if and only if this is adjacent to some node in $C$. That is, $\G_{[C]} = \left([C] \cup (N \setminus C),E_{[C]}\right)$ where $E_{[C]} = E[N \setminus C] \cup \left\{\{[C],j\}:\, \{i,j\} \in E \text{ with } i \in C\right\}$.


\section{Myerson Interaction Index}
\label{sec:MII}

Following \cite{My77}, we consider now a \emph{communication situation} given by a triplet $(N,v,\G)$, being $(N,v)$ a superadditive TU game, and $\G=(N,E)$ an undirected communication graph without loops or parallel arcs, describing the cooperation opportunities available to the players. Let $CS^N$ denote the class of all communication situations with player set $N$.

Myerson's basic idea was that players may cooperate by forming a series of bilateral agreements that are only possible among pairs of linked players. Thus, only connected coalitions can arise. In this setting, identifying all possible cooperation structures on $N$ with $\cG^N$ (the set of all graphs with node set $N$), he defined an \emph{allocation rule} for game $v$ as any function $\Upsilon:\cG^N \to \mathbb{R}^N$ such that $\forall\, \G \in \cG^N$ holds
\begin{equation} \label{eq:alloc_comp_effic}
  \sum_{i \in C} \Upsilon_i(\G) = v(C),
  \quad \forall\, C \in \cK_{\G}.
\end{equation}

Moreover, he defined the \emph{restricted game} $v^{\G}$ of the game $v$ by the graph $\G$, which is given by 
\begin{equation*}
  v^{\G}(S) \coloneq \sum_{C \in \cK_{\G}(S)} v(C),
  \quad \forall S \subseteq N;
\end{equation*}
and what is now called the \emph{Myerson allocation rule} as the Shapley value of the graph restricted game, which is the allocation rule $\mu^v:\cG^N \rightarrow \SR^N$ for game $v$ defined by
\begin{equation*}
  \mu_i^v(\G) \coloneq \phi_i(v^{\G}), \,\forall\, i\in N.
\end{equation*}

Now, following Myerson, we extend interaction indices to communication situations as follows.
\begin{definition}[Graph interaction index]
  Given a game $v \in G^N$, we define a \emph{graph interaction index} for game $v$ to be any function $GI^v:\cG^N \rightarrow \cI^v$, where $\cI^v$ is the class of all interaction indices of game $v$, such that $\forall\, \G \in \cG^N$ satisfies
  \begin{equation}\label{eq:inter_comp_indep}
    \begin{array}{c}
      GI^{v,\G}(S) = 0, \quad\forall S \subseteq N \text{ such that } \nexists C \in \cK_{\G} \text{ with } S \subseteq C \text{, and} \\[2mm]
      GI^{v,\G}(S) = GI^{v\vert_C,\G_C}(S), \quad\forall S \subseteq C \text{ with } C \in \cK_{\G},
    \end{array}
  \end{equation}
  where $GI^{v,\G}(S)$ stands for $[GI^v(\Gamma)](S)$ and $v\vert_C$ for the game $v$ restricted to the set of players $C$.
\end{definition}

Note that condition \eqref{eq:inter_comp_indep} imposes that there is no interaction between groups of individuals belonging to unconnected worlds, which is in accordance with condition \eqref{eq:alloc_comp_effic} imposing no externality between groups of players in different connected components.

Consequently, we extend the Shapley Interaction Index to situations with restricted communication as follows.
\begin{definition}[Myerson interaction index]
  Given a game $v \in G^N$, we define the \emph{Myerson Interaction Index} for $v$ as the graph interaction index $MI^v$ determined $\forall\, \G \in \cG^N$ by $MI^{v,\G}(S) \coloneq SI^{v^{\G}}(S)$, $\forall\, S\subseteq N$. 
\end{definition}

Let us check that the Myerson  interaction index, $MI$, defined above verifies condition \eqref{eq:inter_comp_indep} and, accordingly, that it is indeed a graph interaction index. 

\begin{proposition} \label{prop:MII_gii}
  Let $v \in G^N$ be a game, then MII is a graph interaction index. In particular, given graph $\G \in \cG^N$, it is verified:
  \begin{enumerate}[label=(\roman*)]
    \item If $S$ is not contained in a unique connected component of $\G$, then $MI^{v,\Gamma}(S) = 0$.
    \item If $S \subseteq C$ for some $C \in \cK_{\G}$, then $MI^{v,\Gamma}(S) = MI^{v\vert_C,\G_C}(S)$.
  \end{enumerate}
\end{proposition}
\begin{proof}
  \vphantom{a}
  \begin{enumerate}[label=(\roman*)]
    \item This follows from expression $\eqref{eq:SII_def_divid}$ of the SII. It holds that
          \begin{equation*}
            MI^{v,\G}(S) = SI^{v^{\G}}(S) = \sum_{T \supseteq S} \frac{\Delta_{v^{\G}}(T)}{(t-s+1)} = 0,
          \end{equation*}
          since every $T$ containing $S$ must be disconnected being $S$ not contained in a unique connected component and, thus, $\Delta_{v^{\Gamma}}(T) = 0$, $\forall\, T \supseteq S$, according to Theorem 2 and expression (8) in \cite{Ow86}.
    \item We consider again expression $\eqref{eq:SII_def_divid}$ of the SII. In this case
          \begin{equation*}
            MI^{v,\G}(S) = SI^{v^{\G}}(S) = \sum_{T \supseteq S} \frac{\Delta_{v^{\G}}(T)}{(t-s+1)} = \sum_{T \supseteq S,\ T \subseteq C} \frac{\Delta_{v^{\G}}(T)}{(t-s+1)},
          \end{equation*}
          since every $T$ containing $S$ not contained in $C$ must be disconnected and, thus, $\Delta_{v^{\G}}(T) = 0$, $\forall\, T \supset C$, in line with the previous argument. Within the set of players $C$, it holds that $v^{\G}(T) = [v\vert_C]^{\G_C}(T)$, $\forall\, T \subseteq C$. Therefore, $\Delta_{v^{\G}}(T) = \Delta_{[v\vert_C]^{\G_C}}(T)$, $\forall\, T \subseteq C$, which results in
          \begin{equation*}
            MI^{v,\G}(S) = \sum_{T \supseteq S,\ T \subseteq C} \frac{\Delta_{[v\vert_C]^{\G_C}}(T)}{(t-s+1)} = SI^{[v\vert_C]^{\G_C}}(S) = MI^{v\vert_C,\G_C}(S).
          \end{equation*}
  \end{enumerate}
\end{proof}

Note that the dividends $\Delta_{v^\G}(S)$ vanish whenever $S$ is disconnected in $\G$, and they depend solely on the induced subgraph $\G_S$. Hence, by using the expression of $(N,u_S^\G)$ given in Proposition~2.4 of \cite{GoGoMaOwPoTe03} together with expression (8) in \cite{Ow86}, the non-zero dividends of the graph-restricted game can be written in terms of the dividends of the original game as follows\footnote{For clarity, we assume that $\G$ is connected. Otherwise, replace $\G$ by $\G_C$, where $C$ is the connected component of $\G$ containing $S$.}:
\begin{equation}
\Delta_{v^\G}(S)
= \sum_{\displaystyle\substack{T\subseteq S \\ S=\cup_{m\in M_T(S)} T_m}}
(-1)^{|M_T(S)|+1} \,\Delta_v(T),
\label{relacion-dividendos-general}
\end{equation}
where $\cM_{\G}(T)=\{T_1,\dots,T_{\ell_T}\}$ denotes the family of minimal $T$-connecting subgraphs, and $M_T(S)\subseteq \{1,\dots,\ell_T\}$ is the set of indices of those minimal sets whose union equals $S$, i.e., $S=\cup_{m\in M_T(S)} T_m$. In fact, we can restrict to $\cM_{\G_S}(T)$.

Thus, $\Delta_{v^\G}(S)$ is a linear combination of the dividends of the original game associated with those coalitions that $S$ connects in a minimal way.

If $\G$ is a tree, the family $\cM_{\G}(T)$ contains a unique element for every $T$, and expression~\eqref{relacion-dividendos-general} reduces to
\begin{equation}
\Delta_{v^\G}(S)
= \sum_{L \subseteq \cC(S)} \Delta_v\big(S\setminus L\big),
\label{relacion-dividendos-tree}
\end{equation}
where $\cC(S)$ is the set of cutnodes of $S$, i.e., those nodes in $S$ whose removal disconnects $S$.


\subsection{Properties}
\label{subsec:MII_prop}

In order to propose interesting properties which consider simultaneously games on different player sets we extend the notion of graph interaction index, from an interaction rule for a given game $v \in G^N$ to an index defined over the class of all TU games with a finite set of players. Formally, let $\cU = \{1,2,\dots\}$ be the \emph{universe of players} and $\cN$ the class of all non-empty finite subsets of $\cU$. In addition, let $G = \bigcup_{N \in \cN}{G^N}$ be the set of all characteristic functions. We extend the notion of \emph{graph interaction index} $GI$ to be an assignation rule which associates each game $(N,v) \in G$ with a graph interaction index for $v$, i.e. with a function $GI^v:\cG^N \rightarrow \cI^v$ satisfying condition \eqref{eq:inter_comp_indep}.

First, some usual definitions such as that of dummy player and partnership \citep{KaSa87} for TU games must be properly adapted to this new setting to incorporate the effect of the communication structure.

Note that it has been assumed that there is no interaction between groups of players in different connected components. Therefore, in the sequel, we will restrict our attention to connected graphs without loss of generality. If the underlying communication graph $\G$ is disconnected, definitions and results apply to each connected component.

    \begin{definition}[Graph null  player]
  Let $(N,v,\G)$ be a communication situation. Then, $i \in N$ is a \emph{graph  null  player} in $(N,v,\G)$ if and only if:
  \begin{itemize}
      \item[$(GNi)$] $i$ is a {\it null  player} in the original game $v$, i.e. $v(S \cup i) = v(S)$, $\forall\, S \subseteq N \setminus i$; and
         \item[$(GNii)$] $i$ is a {\it null  node} in the communication graph $\Gamma$, i.e. for every pair $j\neq k\in N\setminus i$ with $i\in B_\G(\{j,k\})$, either $j\in D_0(v)$ or $k\in D_0(v)$, where $D_0(v)\subseteq N$ denotes the set of null  players of the game $v$.
  \end{itemize}
  \label{definicion-graph-dummy-fuerte}
\end{definition}
Let $D_0^{\G}(v) \subseteq N$ denote the set of all graph null  players in communication situation $(N,v,\G)$. 
 In line with the definition of \emph{superfluous player} in a graph game as a null player in the restricted game provided by \cite{SlNo12}, we further develop the above definition of graph null  player and its relation with null  players of the graph restricted game $(N,v^\G)$.

It is easy to show (see Proposition~\ref{PropGraphDummyDummyGRG}) that every player $i\in N$ who is graph null  in $(N,v,\G)$ is also a null  player in the corresponding graph restricted game $(N,v^\G)$. However, being a null  player in $(N,v^\G)$ does not imply to be a graph null . It could be the case, see examples in Example~\ref{examples_graphDummy} below, that $i$ is isolated from other players necessary to generate value, and thus is a null  in $(N,v^\G)$ without being a null  player in $(N,v)$, and thus is not a graph null  player either.

\begin{example} To illustrate the previous definition we may consider some examples which involve null  players from the original game becoming relevant in the graph restricted game and viceversa:
\begin{itemize}
  \item 
    Consider $N=\{1,2,3,4\}$, and $v(i)=0$, $\forall\, i\in N$, $v(\{1,2\})=v(\{2,3\})=v(\{1,3\})=v(\{1,2,3\})=0$, and $v(S)=2$, for the remaining coalitions, which is a superadditive game. If $\G=\{ \{1,2\}, \{2,3\}, \{3,4\}\}$, then player 2 is null  in $(N,v^\G)$, but it is not graph null  according to Definition \ref{definicion-graph-dummy-fuerte} (player 2 is not null  in the original game $(N,v)$, negating condition $(GNi)$).
      \item Consider $N = \{1,2,3\}$ and the unanimity game of coalition $\{2,3\}$. Clearly, player $1$ is null  in this game. However, if we consider the communication situation associated with graph $\G = \{\{1,2\},\{1,3\}\}$, the graph restricted game becomes the unanimity game of coalition $N$ in which player $1$ is not null  thanks to being an intermediary of the coalition $\{2,3\}$ with positive contribution (negating condition $(GNii)$.
  \item Consider $N = \{1,2,3,4\}$ and the unanimity game of coalition $\{2,3,4\}$. Clearly, player $1$ is the only null  in this game. However, if we consider the communication situation associated with graph $\G = \{\{1,2\},\{2,3\}\}$, the graph restricted game becomes the null game as coalition $\{2,3,4\}$ is unattainable within a connected coalition, which corresponds with every player being isolated from other players necessary to generate value.
  However, note that in this case, since we analyse each connected component independently, all players become graph-null  according to Definition \ref{definicion-graph-dummy-fuerte}, since all them  fulfil conditions $(GNi)$ and $(GNii)$ for the game restricted to the connected component of $\G$ to which each player belongs.
 \end{itemize}  
    \label{examples_graphDummy}
\end{example}


\begin{proposition}
    Let $(N,v,\G)$ be a communication situation. If $i \in N$ is a graph null  player, then it is also a null  player in the graph restricted game $(N,v^\G)$.
    \label{PropGraphDummyDummyGRG}
\end{proposition}

\begin{proof}
Let $i\in N$ be a graph null  node, and $S\subseteq N\setminus i$, then
$$
v^\G(S\cup i)=\sum_{L \subseteq S} \Bigl (\Delta_{v^\G} (L) + \Delta_{v^\G} (L\cup i) \Bigr ).
$$
Now, we will show that $\Delta_{v^\G} (L\cup i)=0$, for all $\emptyset \neq L\subseteq N\setminus i$ and, thus, trivially $v^\G(S\cup i)=v^\G(S) + v^\G (i)=v^\G(S) + v(i)=v^\G(S)$ and the result holds.

Let $\emptyset \neq L\subseteq S$ be a given coalition, then taking into account expression \eqref{relacion-dividendos-general} it follows that $\Delta_{v^\G} (L\cup i)$ is a linear combination of the dividends in the original game $(N,v)$ of coalitions $T\subseteq L\cup i$ with $L\cup i=\cup_{m \in M_T(L\cup i)} T_m$ for some $M_T(L\cup i)\subseteq \{1,\dots, \ell_T\}$, being $\cM_{\G}(T) = \{T_1,\dots,T_{\ell_T}\}$ the collection of minimal $T$-connecting subgraphs. Then, one of the two following cases must hold:
\begin{itemize}
    \item[$(i)$] $i\in T$, then $\Delta_v(T)=0$ since $i$ is a null  player in $(N,v)$ and $\vert T\vert \geq 2$ (otherwise $\cM_{\G}(T) =\{ T\}\neq L\cup i$, since $L\neq \emptyset$).
      \item[$(ii)$]  $i\notin T$, then there exist $j,k\in T$, $j\neq k$ with $i\in B_\G(\{j,k\})$. Thus $\Delta_v(T)=0$ since $j\in D_0(v)$ or $k\in D_0(v)$ and $\vert T\vert \geq 2$.
\end{itemize}
    \end{proof}

With regard to partnerships $P\subseteq N$ in the original game, we are interested on a special kind of graph partnerships which we refer to as {\it veto graph partnerships}, where the remaining players are unable to generate value without them. Thus, it should be noted that all players who are strictly necessary to connect them can be included in the partnership when the communication structure is taken into account. Formally:

\begin{definition}[Veto Graph Partnership] Let $(N,v,\Gamma)$, then a non-empty $GP\subseteq N$ is a {\it veto graph partnership} if there exists a {\it veto partnership} $P\neq \emptyset$ in $(N,v)$, i.e. $v(T\cup S)=v(T)=0$, for all  $T\subseteq N\setminus P$, and all $S\subsetneq P$, such that $\emptyset \neq GP\subseteq P\cup  EB_\Gamma(P)$.
\end{definition}

Members of $P$ behave in the game like a single hypothetical player, since they are a partnership in the game $(N,v)$. On the other hand, all members of $L\subseteq EB_\G(P)$ are indispensable to connect $P$ and contribute nothing in the game $(N,v)$ if some player of $P$ is not present.
Thus, in the communication structure the members of $P\cup L$ act as a single hypothetical player and therefore the original game should be equivalent to the game in which the coalition $P\cup L$ is replaced by a player acting as its proxy. It is worth noting, however, that being a veto graph partnership does not mean that the structure of the communication will be changed.

Now, with regard to graph partnerships, we could consider its natural relation with partnerships in the graph-restricted game, as we have done with graph dummy players. We will show that being a veto graph partnership implies being a veto partnership in the graph-restricted game. However, being a veto partnership in $(N,v^\G)$ does not imply to be a veto graph partnership (see Example~\ref{CEx_partner})

\begin{proposition}
    Let $(N,v,\G)$ be a communication situation. If $GP\subseteq N$ is a veto graph partnership, then it is also a partnership in the graph restricted game $(N,v^\G)$.
    \label{PropVetoGraphPartnershipGRG}
\end{proposition}

\begin{proof}
Let $\emptyset \neq GP\subseteq P\cup EB_\G(P)$, for some veto partnership $P$ of $(N,v)$. Then, for any $T\subseteq N\setminus GP$, and  all $S\subsetneq GP$, it is verified:
$$
v^\G(T\cup S)=\sum_{C\in {\cal K}_\G(T\cup S)} v(C)=\sum_{C\in {\cal K}_\G(T\cup S)} 
v((C\cap P)\cup (C\setminus P)),
$$
with $(C\cap P)\subsetneq P$ (otherwise $P\cup EB_\G(P)\subseteq T\cup S$). Thus, $v((C\cap P)\cup (C\setminus P))=0$, for all $C\in {\cal K}_\G(T\cup S)$ and $GP$ is also a veto partnership in $(N,v^\G)$.
\end{proof}

\begin{example} Consider the 4 player superadditive game $(N,v)$ with $v(\{1,3\}) = v(\{1,4\}) = v(\{1,2,3\}) = v(\{1,2,4\}) = v(\{1,3,4\}) = v(N) = 1$, where
there are not veto partnerships. If we consider $\G = \{\{1,2\},\{2,3\},\{3,4\}\}$, then $\{1,2,3\}$ is a veto partnership in the graph restricted game $(N,v^\G)$, but is not a veto graph partnership.
    \label{CEx_partner}
\end{example}

We may now propose relevant properties for the characterisation of the MII.

\begin{definition} 
  Let $GI$ be a graph interaction index defined over the set of all finite games $G$ that associates to each game $(N,v) \in G$ the graph interaction index $GI^v$ given by $[GI^v(\G)](S) \equiv I^{v,\G}(S)$, $\forall S \subseteq N$. Then, $GI$ verifies:
  \begin{itemize}
    \item[($I$-CE)] \emph{$I$-Component Efficiency}, if and only if
         \begin{equation*}
           \sum_{i \in C} {I^{v,\G}(i)} = v(C),\quad\forall C \in \cK_{\G},
         \end{equation*}
         for any game $(N,v) \in G$ and graph $\G \in \cG^N$;

  \item[($I$-GN)] \emph{$I$-Graph Null}, if and only if
         \begin{equation*}
             I^{v,\G}(S \cup i) = 0,\quad\forall S \subseteq N \setminus i, 
         \end{equation*}
         for any graph null  player $i \in D_0^\G(v)$, game $(N,v) \in G$ and graph $\G \in \cG^N$;

    \item[($I$-F)] \emph{$I$-Fairness}, if and only if
         \begin{equation*}
           I^{v,\G}(S \cup i)-I^{v,\G_{-\{i,j\}}}(S \cup i) = I^{v,\G}(S \cup j)-I^{v,\G_{-\{i,j\}}}(S \cup j),\quad\forall S \subseteq N \setminus \{i,j\},
         \end{equation*}
         for any edge $\{i,j\} \in E$, game $(N,v) \in G$ and graph $\G = (N,E) \in \cG^N$, where $\G_{-\{i,j\}} = (N,E \setminus \{i,j\})$;

    \item[($I$-SRVPC)] \emph{$I$-Strong Reduced Veto Partnership Consistency}, if and only if
          \begin{equation*}
           I^{v,\G}(P \cup S) = I^{v^\G_{[P]},\G_{[P]}}([P] \cup S),\quad\forall S \subseteq N \setminus P,
         \end{equation*}
         for any veto graph partnership $P \subseteq N$, game $(N,v) \in G$ and graph $\G \in \cG^N$, where $v^\G_{[P]}$ and $\G_{[P]}$ are, respectively, the quotient of the graph restricted game\footnote{$v^\G_{[P]} (S)=v^\G(S)$, and $v^\G_{[P]} (S\cup [P])=v^\G(S\cup P)$, $\forall\, S\subseteq N\setminus P$. Note that $v^\G_{[P]}=v^{\G_{[P]}}$ if $P$ is connected in $\G$.} and quotient graph with respect to $P$;    
    \item[($I$-L)] \emph{$I$-Linearity}, if and only if
         \begin{equation*}
           I^{\alpha_1\,v_1+\alpha_2\,v_2,\G}(S) = \alpha_1\,I^{v_1,\G}(S) + \alpha_2\,I^{v_2,\G}(S),\quad\forall S \subseteq N,
         \end{equation*}
         for any weights $\alpha_1, \alpha_2 \in \SR$, games $(N,v_1), (N,v_2) \in G$ with same player set $N$ and graph $\G \in \cG^N$.
\end{itemize}
\end{definition}

$I$-CE is based on the strong component efficiency by \cite{My77} and the idea that the interaction of singletons should be their Myerson value, which \cite{GrRo99} proof for the SII with respect to the Shapley value. Recall that any graph interaction index verifies that there is no interaction between individuals belonging to unconnected worlds, so that $I^{v,\G}(i) = I^{v\vert_C,\G_C}(i)$, $\forall i \in C$ with $C \in \cK_{\G}$ (as per condition \ref{eq:inter_comp_indep}).

Property $I$-GN extends the dummy player axiom considered in \cite{GrRo99} and \cite{FuKoMa06} to the case of restricted communication, while requiring it only for those dummy players that are also null players.

$I$-F, which together with $I$-SRVPC plays a crucial role in the axiomatic characterisation of the Myerson interaction index, generalises the property of fairness that characterises the Myerson allocation rule. It states the equal-gains principle of equity considered by \cite{My77} -- ``two players should gain equally from their bilateral agreement'' -- over the interaction of these two players $i$ and $j$ with any other group $C \subseteq N \setminus \{i,j\}$. Note that the original property is a particular case of $I$-F for $C = \emptyset$.

$I$-SRVPC is a slightly stronger version of the reduced partnership consistency axiom introduced in \cite{FuKoMa06} to characterise the Shapley and Banzhaf interaction indices. As the authors argue -- ``the interaction among the players of a partnership $P$ in a game $v$ should be regarded as the value of the reduced partnership $[P]$ in the corresponding quotient game $v_{[P]}$''.  
We extend the same reasoning to the interaction of players in a veto graph partnership $P$ with the remaining players in $N \setminus P$, whenever the partnership was a veto one.

$I$-L is the usual linearity axiom considered by \cite{GrRo99}, and also by \cite{FuKoMa06}.

\begin{proposition} \label{prop:MII_propert}
  The Myerson interaction index satisfies all properties $I$-CE, $I$-GN, $I$-F, $I$-SRVPC and $I$-L.
\end{proposition}

\begin{proof}
  Let us arbitrarily fix the player set $N$, game $v \in G^N$ and graph $\G \in \cG^N$.
  \begin{itemize}
    \item[$I$-CE:] First note that, from expression \eqref{eq:SII_def_deriv} of the SII, it is straightforward that $SI^v(i)$ is the Shapley value of player $i$ in game $v$, $\forall i \in N$. Therefore, $SI^{v^{\G}}(i)$ is the Shapley value of player $i$ in game $v^{\G}$, i.e. it is its Myerson value $\mu_i^v(\G) = \phi_i(v^{\G})$, which verifies condition \eqref{eq:alloc_comp_effic}. Thus, $\forall C \in \cK_{\G}$, it follows
    \begin{equation} \label{eq:comp_consist} 
      \sum_{i \in C} {MI^{v,\G}(i)} = \sum_{i \in C} {SI^{v^{\G}}(i)} = \sum_{i \in C} {\phi_i(v^{\G})} = \sum_{i \in C} {\mu_i^v(\G)} = v(C).
    \end{equation}
    Hence, $I$-Component Efficiency holds.

    \item[$I$-GN:] Now, note that every graph  null player is a  null player in the graph restricted game $v^\G$ and SII verifies dummy axiom \citep{GrRo99}. Thus, $\forall i \in D_0^\G(v)$, it follows
    \begin{equation*}
      \begin{array}{c}
        MI^{v,\G}(S \cup i) = SI^{v^{\G}}(S \cup i) = 0,\quad\forall S \subseteq N \setminus i, S \neq \emptyset \text{, and} \\[2mm]
        MI^{v,\G}(i) = SI^{v^{\G}}(i) = v^{\G}(i) = v(i)=0.
      \end{array}
    \end{equation*}
    Hence, $I$-Graph null  holds.
    
    \item[$I$-F:] Recall expression $\eqref{eq:SII_def_divid}$ of the Shapley Interaction Index:
    \begin{equation*}
      MI^{v,\G}(S) \coloneq SI^{v^{\G}}(S) = \sum_{T \supseteq S} \frac{\Delta_{v^{\G}}(T)}{t-s+1}
    \end{equation*}
    Thus, for any edge $\{i,j\} \in E$ and $S \subseteq N \setminus \{i,j\}$, it holds
    \begin{equation*}
      MI^{v,\G}(S \cup i) - MI^{v,\G_{\NEG{\{i,j\}}}}(S \cup i) = \sum_{T \supseteq S \cup i} \frac{\Delta_{v^{\G}}(T)-\Delta_{v^{\G_{\NEG{\{i,j\}}}}}(T)}{t-s}.
    \end{equation*}
    Note that $\Delta_{v^{\G}}(T) = \Delta_{v^{\G_{\NEG{\{i,j\}}}}}(T)$ for any subset $T$ such that $i \notin T$ or $j \notin T$, as edge $\{i,j\}$ is not relevant in its communication structure. Therefore,
    \begin{equation*}
      MI^{v,\G}(S \cup i) - MI^{v,\G_{\NEG{\{i,j\}}}}(S \cup i) = \sum_{T \supseteq S \cup \{i,j\}} \frac{\Delta_{v^{\G}}(T)-\Delta_{v^{\G_{\NEG{\{i,j\}}}}}(T)}{t-s}.
    \end{equation*}
    Analogously, it holds
    \begin{equation*}
      MI^{v,\G}(S \cup j) - MI^{v,\G_{\NEG{\{i,j\}}}}(S \cup j) = \sum_{T \supseteq S \cup \{i,j\}} \frac{\Delta_{v^{\G}}(T)-\Delta_{v^{\G_{\NEG{\{i,j\}}}}}(T)}{t-s},
    \end{equation*}
    so $I$-Fairness is satisfied by MI.
    
    \item[$I$-SRVPC:]        
    Let $P \subseteq N$ be a veto graph partnership and $S \subseteq N \setminus P$. As before, from expression $\eqref{eq:SII_def_divid}$ of the Shapley Interaction Index follows
    \begin{equation*}
      MI^{v,\G}(P \cup S) = \sum_{T \supseteq P \cup S} \frac{\Delta_{v^{\G}}(T)}{t-s-p+1}.
    \end{equation*}
    Proposition \ref{PropVetoGraphPartnershipGRG} showed that $P$ is a partnership in the graph restricted game $(N,v^{\G})$, so $\Delta_{v^{\G}}(L \cup P) = \Delta_{v_{[P]}^{\G}}(L \cup [P])$ for any subset $L \subseteq N\setminus P$. In addition, note that $|L \cup [P]| = |L \cup P| - p + 1$. Therefore,
    \begin{equation*}
      MI^{v,\G}(P \cup S) = \sum_{T \supseteq [P] \cup S} \frac{\Delta_{v_{[P]}^{\G}}(T)}{t-s} = MI^{v^\G_{[P]},\G_{[P]}}([P] \cup S).
    \end{equation*}
    Hence, $I$-Strong Reduced Veto Partnership Consistency holds.

    \item[$I$-L:] Finally, note that SII verifies linearity \citep{GrRo99}. Since the graph restriction operation is linear, $\forall\alpha_1, \alpha_2 \in \SR$, $\forall v_1, v_2 \in G^N$, it follows
    \begin{equation*}
      \begin{array}{r c l}
        MI^{\alpha_1\,v_1+\alpha_2\,v_2,\G}(S) & = & SI^{(\alpha_1\,v_1+\alpha_2\,v_2)^{\G}}(S) \\
        & = & \alpha_1\,SI^{v_1^\G}(S) + \alpha_2\,SI^{v_2^\G}(S) \\
        & = & \alpha_1\,MI^{v_1,\G}(S) + \alpha_2\,MI^{v_2,\G}(S),
      \end{array}
      \quad\forall S \subseteq N.
    \end{equation*}
    Hence, $I$-Linearity holds.
\end{itemize}

\end{proof}

\begin{theorem}\label{th:MII_charact}
The unique graph interaction index over the set of all communication structures with finite games in $G$ verifying $I$-CE, $I$-F,$I$-GN, $I$-L and $I$-SRVPC is the Myerson interaction index MI.

\end{theorem}

\begin{proof}
We have proved that the properties hold for the Myerson interaction index $MI$ in Proposition~\ref{prop:MII_propert}. Therefore, we are left with the question of \textbf{uniqueness}.

Let $GI$ be a graph interaction index verifying $I$-CE, $I$-F, $I$-GN, $I$-L and $I$-SRVPC. First, we prove the uniqueness for first-order interactions. Fix a player set $N \in \cN$ and a game $(N,v)\in G^N$. By $I$-CE and $I$-F, the allocation rule defined by
\begin{equation*}
  \Upsilon_i(\G) = GI^{v,\G}(\{i\}), \quad\forall i \in N,
\end{equation*}
satisfies component efficiency and fairness. Hence, it must coincide with the Myerson value, which is itself the first-order Myerson interaction.

To prove the uniqueness for higher-order interactions, we rely on $I$-linearity. Since the family $\{(N,u_S):\,S \subseteq N, S \neq \emptyset\}$ forms a basis of $G^N$ for all $N \in \cN$, it suffices to consider unanimity games $(N,u_T)$ with $T \in 2^N \setminus \{\emptyset\}$. Thus, we will show that
\begin{equation} \label{resultado-interaccS-unanimT}
  GI^{(N,u_T),\G}(S) = MI^{(N,u_T),\G}(S) \coloneq SI^{(N,u_T^\G)}(S),
\end{equation}
for all $S \subseteq N$ with $|S| \geq 2$ and $N \in \cN$.

If $|T| = 1$, that is, for dictatorial games $(N,u_{\{i\}})$, every player $j \neq i$ is graph null. Hence, by $I$-GN,
\begin{equation*}
  GI^{(N,u_{\{i\}}),\G}(S) = 0 = MI^{(N,u_{\{i\}}),\G}(S),\quad \forall S \subseteq N \text{ with } |S| \geq 2,
\end{equation*}
because $S \neq \{i\}$ in any case. Thus, \eqref{resultado-interaccS-unanimT} holds. 

Assume now that $|T| \geq 2$. We proceed by complete induction on the cardinality of the player set $N \in \cN$.

First, we prove \eqref{resultado-interaccS-unanimT} whenever $|N| = 2$, so we only need to check uniqueness for second-order interactions ($|S| = 2$) and the grand coalition unanimity game ($T = N$). Let $(N,u_N)$ be a two-player unanimity game with $N = \{i,j\} = S$. Then $\cG^N = \{\G_1,\G_2\}$, where $\G_1 = (N,E_1)$ with $E_1 = \emptyset$ and $\G_2 = (N,E_2)$ with $E_2 = \{\{i,j\}\}$.
\begin{itemize}
  \item In the case $(N,u_N,\G_1)$, it follows
  \begin{equation*}
    GI^{(N,u_N),\G_1}(\{i,j\}) = 0 = MI^{(N,u_N),\G_1}(\{i,j\})
  \end{equation*}
  by condition \eqref{eq:inter_comp_indep} in the definition of graph interaction indices, as $S = \{i,j\}$ is not connected in $\G_1$. 
  \item In the case $(N,u_N,\G_2)$, the set $N$ is a veto graph partnership. By $I$-SRVPC and uniqueness of first-order interactions, it follows
  \begin{equation*}
    GI^{(N,u_N),\G_2}(N) = GI^{([N],u_{[N]}),\G_{2,[N]}}([N]) =
    MI^{([N],u_{[N]}),\G_{2,[N]}}([N]) = MI^{(N,u_N),\G_2}(N).
  \end{equation*}
  Note that $([N],u^{\G_2}_{[N]}) \equiv ([N],u_{[N]})$, since $\G_2$ is the complete graph of $N$. 
\end{itemize}

Now, let $N \in \cN$ with $|N| = n > 2$, and assume that \eqref{resultado-interaccS-unanimT} holds for all player sets $N'$ with $|N'| < n$. By the induction hypothesis and the results for first-order interactions and dictatorial games, for every unanimity game $(N',u_T)$ with $|N'| < n$,
\begin{equation*}
  GI^{(N',u_T),\G}(S) = MI^{(N',u_T),\G}(S) \coloneq SI^{(N',u_T^\G)}(S), \quad
  \forall S \in 2^{N'} \setminus \{\emptyset\} \text{ and } \G \in \cG^{N'}.
\end{equation*}
Hence, by $I$-L, it holds
\begin{equation}\label{hip-induc-vgeneral}
  GI^{(N',v),\G}(S) = MI^{(N',v),\G}(S) \coloneq SI^{(N',v^\G)}(S),
\end{equation}
for all games $(N',v)\in G^{N'}$. Let us prove \eqref{resultado-interaccS-unanimT} for all non-empty $T,S \subseteq N$ with $|T|,|S| \geq 2$.

Fix now $T \subseteq N$ with $|T| \geq 2$ and a graph $\G \in \cG^N$. If $\G$ is not connected, condition \eqref{eq:inter_comp_indep} implies the following:
\begin{itemize}
  \item If there is no $C \in \cK_{\G}$ such that $T \subseteq C$, then $GI^{u_T,\G}(S) = 0 = MI^{u_T,\G}(S)$ for all $S \in 2^N \setminus \{\emptyset\}$.
  \item Otherwise, let $C_T$ be the connected component of $\G$ containing $T$. Given $S \in 2^N \setminus \{\emptyset\}$, then $GI^{u_T,\G}(S) = 0 = MI^{u_T,\G}(S)$ for all $S \not\subseteq C_T$, and
  \begin{equation*}
    GI^{u_T,\G}(S) = GI^{u_T\vert_{C_T},\G_{C_T}}(S) = MI^{u_T\vert_{C_T},\G_{C_T}}(S) = MI^{u_T,\G}(S),
  \end{equation*}
  for all $S \subseteq C_T$, where \eqref{hip-induc-vgeneral} links $GI$ and $MI$.
\end{itemize}
Thus, it suffices to consider connected graphs $\G$. 

Let $S \subseteq N$ with $|S| \geq 2$. We show that
\begin{equation}\label{probarST}
  GI^{(N,u_T),\G}(S) = MI^{(N,u_T),\G}(S).
\end{equation}
For simplicity, as $T$ is fixed, henceforth we will write $v = u_T$.

\noindent\textbf{Case 1.}
If $S \not\subseteq T \cup B_{\G}(T)$, then $S$ contains graph null players and, by $I$-GN, $GI^{v,\G}(S) = 0 = MI^{v,\G}(S)$.

Subsequently, we will assume $S \subseteq T \cup B_{\G}(T)$. Let $P \coloneq T \cup EB_{\G}(T)$, which forms a veto graph partnership, and denote $S_P = S \cap P$ and $S_{\NEG{P}} = S \setminus P$.

\noindent\textbf{Case 2.}
If $|S_P| \geq 2$, by $I$-SRVPC with respect to veto graph partnership $S_P \subseteq P$ it holds
\begin{equation*}
  GI^{v,\G}(S) = GI^{v^\G_{[S_P]},\G_{[S_P]}}([S_P]\cup S_{\NEG{P}}) =
  MI^{v^\G_{[S_P]},\G_{[S_P]}}([S_P]\cup S_{\NEG{P}}) = MI^{v,\G}(S),
\end{equation*}
where equality between $GI$ and $MI$ follows from the induction hypothesis \eqref{hip-induc-vgeneral} since $|[N \setminus S_P] \cup \{[S_P]\}| \leq n-1$.

\noindent\textbf{Case 3.}
If $S_P = \{t_I\}$ with $t_I \in T \cup EB_{\G}(T)$, so that $|S_P| = 1$, in order to prove \eqref{probarST} we proceed by a secondary induction, relating this case to the previous one. 

The argument is based on a path construction that allows one to relate the interaction of the coalition $S$ to that of a coalition of the form $S \setminus \{j\} \cup \{t_O\}$, where $j \in S$ is a non-essential intermediary of $T$ and $t_O \in P$ is an essential player. Along such a path, the axioms $I$-GN and $I$-F generate a square linear system of equations whose solution is unique and results in the desired equality. The proof proceeds by a complete induction on the number $k$ of cycles in graph $\G$.

We first prove that $GI^{v,\G}(S) = MI^{v,\G}(S)$, for every $S \subseteq T \cup B_{\G}(T)$ with $S_P = \{t_I\}$ and $|S| \geq 2$, whenever $\G$ has a unique cycle ($k = 1$). Note that $B_{\G}(T) = EB_{\G}(T)$ in the case $k = 0$, so that $S = S_P$ and $|S_P| = 1$ is incompatible with $|S| \geq 2$.

Since $|S|,|T| \geq 2$ and $S_P = \{t_I\}$, $\exists j' \in S_{\NEG{P}}$ that is a non-essential intermediary of $T$ (because $S \subseteq T \cup B_\G(T)$). Hence, $\exists t_1,t_2 \in T \cup EB_{\G}(T) = P$ such that the minimal path $\mu[t_1,t_2] = (t_1,\ldots,j',\ldots,t_2)$ connecting $t_1$ and $t_2$ through $j'$ consists only of non-essential intermediaries of $T$ with the exception of both endpoints (possibly a subpath of a minimal path between two nodes of $T$ for which $j'$ intermediates). Without loss of generality, assume $t_I \notin \mu[j',t_2]$, and let $j \in S_{\NEG{P}}$ be the last node of $S$ along the path $\mu[j',t_2]$ and $t_O = t_2$. Then, we have determined a minimal path $\mu[j,t_O] = (i_1=j,\ldots,i_r=t_O)$ connecting $j$ and $t_O$ where all nodes except $j$ do not belong to $S$ and all vertices except $t_O$ are non-essential intermediaries of $T$. Note that $\mu[j,t_O]$ is a subpath of a path $\mu[t_1,t_2]$ connecting $t_1$ and $t_2$ where all intermediaries are non-essential. Thus, there must be an alternative path connecting $t_1$ and $t_2$ which does not include any of the intermediaries in $\mu[t_1,t_2]$. This implies that $\mu[j,t_O]$ is part of the unique cycle in $\G$.

Fix such $j \in S_{\NEG{P}}$ and $t_O \in P \setminus S$, and let $E[\mu[j,t_O]] = \{e_1,\ldots,e_{r-1}\}$ where $e_h = \{i_h,i_{h+1}\}$ $\forall 1 \leq h \leq r-1$. Consider $e_h$ with $1 \leq h \leq r-2$. When $e_h$ is removed, $i_h$ and $i_{h+1}$ become graph-null players. Obviously, none of them can become essential intermediaries when removing one of their incident edges. Moreover, if they remained non-essential intermediaries of $T$, there would exist a cycle in $\G_{-e_h}$ containing them, which contradicts the existence of a unique cycle in $\G$ since $e_h$ belongs to it as previously noted\footnote{ In fact, every non-essential intermediary $i_1,\ldots,i_{r-1}$ becomes graph-null when an edge in $\mu[j,t_O]$ is removed.}. Analogously, when $e_{r-1}$ is removed, $i_{r-1}$ becomes a graph-null player.

Therefore, by $I$-GN, it holds
\begin{align*}
  GI^{v,\G_{-e_h}}(S \setminus \{j\} \cup \{i_h\}) & = 0 = MI^{v,\G_{-e_h}}(S \setminus \{j\} \cup \{i_h\}), \\
  GI^{v,\G_{-e_h}}(S \setminus \{j\} \cup \{i_{h+1}\}) & = 0 = MI^{v,\G_{-e_h}}(S_{-j} \cup \{i_{h+1}\}),
\end{align*}
for all $1 \leq h \leq r-2$. Knowing this, and applying $I$-F to each edge $e_h$, we obtain equations
\begin{equation}\label{th:eq_first}
  \begin{array}{r}
    GI^{v,\G}(S \setminus \{j\} \cup \{i_h\}) - GI^{v,\G}(S \setminus \{j\} \cup \{i_{h+1}\}) = 0, \\
    MI^{v,\G}(S \setminus \{j\} \cup \{i_h\}) - MI^{v,\G}(S \setminus \{j\} \cup \{i_{h+1}\}) = 0,
  \end{array}
\end{equation}
for every $1 \leq h \leq r-2$.

For the last edge $e_{r-1} = \{i_{r-1},t_O\}$, $I$-GN yields
\begin{equation*}
  GI^{v,\G_{-e_{r-1}}}(S \setminus \{j\} \cup \{i_{r-1}\}) = 0 = MI^{v,\G_{-e_{r-1}}}(S_{-j} \cup \{i_{r-1}\}).
\end{equation*}
Note that $S \setminus \{j\} \cup \{t_O\}$ contains a two-player veto graph partnership in $(N,v,\G)$, namely $\{t_I,t_O\}$. Then, according to Case~2, it follows
\begin{equation*}
  GI^{v,\G}(S \setminus \{j\} \cup \{t_O\}) = MI^{v,\G}(S \setminus \{j\} \cup \{t_O\}) \eqcolon \sigma_O.
\end{equation*}
Moreover,
\begin{equation*}
  GI^{v,\G_{-e_{r-1}}}(S \setminus \{j\} \cup \{t_O\}) = MI^{v,\G_{-e_{r-1}}}(S \setminus \{j\} \cup \{t_O\}) \eqcolon \sigma_O^-,
\end{equation*}
since $S \setminus \{j\} \cup \{t_O\}$ also contains a two-player veto graph partnership in $(N,v,\G_{-e_{r-1}})$; otherwise, edge $e_{r-1}$ would be essential to connect $\{t_I,t_O\}$, contradicting the fact that $i_{r-1}$ is a non-essential intermediary of $T$. Altogether, applying $I$-F to edge $e_{r-1}$ results in equations
\begin{equation}\label{th:eq_last}
  \begin{array}{r}
    GI^{v,\G}(S \setminus \{j\} \cup \{i_{r-1}\}) = \sigma_O - \sigma_O^-, \\
    MI^{v,\G}(S \setminus \{j\} \cup \{i_{r-1}\}) = \sigma_O - \sigma_O^-.
  \end{array}
\end{equation}

Equations \eqref{th:eq_first} and \eqref{th:eq_last} in either $GI^{v,\G}(S \setminus \{j\} \cup \{i_h\})$ or $MI^{v,\G}(S \setminus \{j\} \cup \{i_h\})$, with $1 \leq h \leq r-1$, form a square linear system of $r-1$ equations with right-hand-side vector $\mathbf{b}^t = (0,\ldots,0,\sigma_O-\sigma_O^-)$ and upper bidiagonal coefficient matrix with determinant equal to 1. Hence, the solution is unique. In particular, for $i_1 = j$, we may conclude the result $GI^{v,\G}(S) = MI^{v,\G}(S)$ that we wanted to prove.

To conclude the proof of this case, we assume that there are $k > 1$ cycles in graph $\G$ and show \eqref{probarST}. By the $k$-induction hypothesis, we may assume that the result holds whenever the number of cycles in graph $\G$ is strictly smaller than $k$. We will proceed as before, by means of selecting a minimal path $\mu[j,t_O]$ in the previous conditions, i.e. $j \in S_{\NEG{P}}$ and $t_O \in P \setminus \{t_I\}$ such that every other node in the path is a non-essential intermediary of $T$ not in $S$. The same reasoning applies, but in this case the right-hand-side vector of the linear system of equations has non-zero entries
\begin{equation*}
  \begin{array}{r c l}
    b_h & \coloneq &
    GI^{v,\G_{-e_h}}(S \setminus \{j\} \cup \{i_h\}) - GI^{v,\G_{-e_h}}(S \setminus \{j\} \cup \{i_{h+1}\}) \\
    & = & MI^{v,\G_{-e_h}}(S \setminus \{j\} \cup \{i_h\}) - MI^{v,\G_{-e_h}}(S \setminus \{j\} \cup \{i_{h+1}\}),
  \end{array}
\end{equation*}
for all $1 \leq h \leq r-2$, where the equality between $GI$ and $MI$ follows from the $k$-induction hypothesis, since removing edge $e_h$ strictly decreases the number of cycles in graph $\G$. In addition,
\begin{equation*}
  \begin{array}{r c l}
    b_{r-1} & \coloneq &
    \sigma_O - \sigma_O^- + GI^{v,\G_{-e_{r-1}}}(S \setminus \{j\} \cup \{i_{r-1}\}) \\
    & = & \sigma_O - \sigma_O^- + MI^{v,\G_{-e_{r-1}}}(S \setminus \{j\} \cup \{i_{r-1}\}),
  \end{array}
\end{equation*}
which follows from the $k$-induction hypothesis and Case~2 acknowledging that $S \setminus \{j\} \cup \{t_O\}$ contains a two-player veto graph partnership.

\noindent\textbf{Case 4.}
If $S_P = \emptyset$, so that $|S_P| = 0$, an analogous argument to that in Case 3 applies. One considers a minimal path connecting some $j \in S$ to an essential player $t_O \in P = T \cup EB_\G(T)$. As in the previous case, the first $r-2$ equations are obtained by applying $I$-GN and $I$-F, while the last equation follows by reducing this situation to Case~3, where exactly one essential player belongs to the corresponding coalition. 
\end{proof}

The following proposition shows that each property considered in Theorem~\ref{th:MII_charact} is necessary to guarantee the uniqueness of the Myerson Interaction Index.

\begin{proposition}
    The axioms considered in Theorem \ref{th:MII_charact} are logically independent.
\end{proposition}

\begin{proof}
    See Appendix \ref{LIA}.
\end{proof}


\section{Graph Interaction due to the network}
\label{sec:applic}

In the original game $(N,v)$, coalition interaction reflects only the complementarities encoded in the characteristic function $v$. Once a communication graph $\Gamma$ is introduced, cooperation becomes constrained and the resulting interaction pattern may differ substantially. In a communication situation $(N,v,\Gamma)$, the Myerson restricted game $v_{\Gamma}$ represents the value attainable under these constraints, and the interaction index $MI^{v,\Gamma}$ captures both the intrinsic structure of $v$ and the relational structure induced by $\Gamma$.

This motivates isolating the component of interaction that is \emph{induced by the network}. For any non-empty $S \subseteq N$, we define the network-induced interaction as
$$
NI^{v,\Gamma}(S) \coloneq MI^{v,\Gamma}(S) - SI^{v}(S).
$$
The term $NI^{v,\Gamma}(S)$ measures how the communication structure amplifies or restricts the interaction among $S$ members. In analogy with notions such as centrality or social capital, it isolates the contribution of the relational environment from the purely game-theoretic effect of $v$. In particular, $NI^{v,\Gamma}(i)$, $i\in N$, coincides with the {\em game-theoretical centrality} introduced in \cite{GoGoMaOwPoTe03} when symmetric games are considered.

We illustrate this concept with two different games defined on the same communication structure. Let $N=\{1,2,3,4,5\}$ and consider the graph $\Gamma$ depicted below:
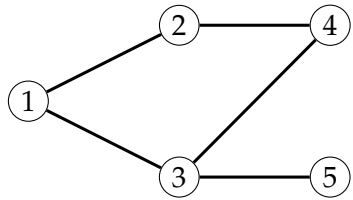
\begin{center}  
\begin{tikzpicture}[node_style/.style={draw,circle,minimum size=0.5cm,inner sep=2}]
			\node[node_style] (s) at (0,1) {$1$} ;
			\node[node_style] (b) at (2,2) {$2$} ;
			\node[node_style] (c) at (2,0) {$3$} ;
			\node[node_style] (d) at (4,2) {$4$} ;
			\node[node_style] (e) at (4,0) {$5$} ;
			
			\draw  [line width=1.15pt] (s) edge (b) ;
			\draw  [line width=1.15pt] (s) edge (c) ;
			\draw  [line width=1.15pt] (b) edge (d) ;
			\draw  [line width=1.15pt] (c) edge (d) ;
			\draw  [line width=1.15pt] (c) edge (e) ;
		\end{tikzpicture}
\end{center}

We consider the following games:
\begin{itemize}
    \item The messages game $v_m(S)=s(s-1)$, for all $S\subseteq N$, where $s=|S|$. This symmetric game captures the ability of coalition members to exchange bilateral messages.
    
    \item A horse market, a canonical example of a pure exchange economy, in which player 1 owns a horse and players 4 and 5 are potential buyers valuing it at 90 and 100 units, respectively. Players 2 and 3 act as intermediaries. The game is given by: $v_h(S)=0$ if $1\notin S$, $v_h(S)=100$ if $\{1,5\}\subseteq S$, and $v_h(S)=90$ if $\{1,4\}\subseteq S$ and $5\notin S$.
\end{itemize}

\begin{table}[h!]
\centering
\setlength{\tabcolsep}{4pt}      
\renewcommand{\arraystretch}{1.1} 

\caption{Interaction values in $(N,v_m)$ for single-player coalitions.}
\label{tab:mess-1}
\begin{tabular}{l*{5}{S[table-format=+2.2]}}
\toprule
 & \multicolumn{5}{c}{Coalition} \\
\cmidrule(lr){2-6}
 & {$\{1\}$} & {$\{2\}$} & {$\{3\}$} & {$\{4\}$} & {$\{5\}$} \\
\midrule
\rowcolor{gray!10}
Myerson & \num{3.77} & \num{3.60} & \num{5.93} & \num{3.77} & \num{2.93} \\
Shapley & \num{4}    & \num{4}    & \num{4}    & \num{4}    & \num{4}    \\
\rowcolor{gray!10}
Network & \num{-0.23}& \num{-0.40}& \num{1.93} & \num{-0.23}& \num{-1.07}\\
\bottomrule
\end{tabular}
\end{table}

\begin{table}
\centering
\setlength{\tabcolsep}{3pt}
\renewcommand{\arraystretch}{1.1}

\caption{Interaction values in $(N,v_m)$ for two-player coalitions }
\label{tab:mess-2}

\begin{tabular}{l*{10}{S[table-format=+2.2]}}
\toprule
 & \multicolumn{10}{c}{Coalition} \\
\cmidrule(lr){2-11}
 & \multicolumn{1}{c}{$\{1,2\}$}
 & \multicolumn{1}{c}{$\{1,3\}$}
 & \multicolumn{1}{c}{$\{1,4\}$}
 & \multicolumn{1}{c}{$\{1,5\}$}
 & \multicolumn{1}{c}{$\{2,3\}$}
 & \multicolumn{1}{c}{$\{2,4\}$}
 & \multicolumn{1}{c}{$\{2,5\}$}
 & \multicolumn{1}{c}{$\{3,4\}$}
 & \multicolumn{1}{c}{$\{3,5\}$}
 & \multicolumn{1}{c}{$\{4,5\}$} \\
\midrule
\rowcolor{gray!10}
Myerson & \num{2.83} & \num{3.83} & \num{1.33} & \num{0.50}
        & \num{1.50} & \num{2.83} & \num{0.83}
        & \num{3.83} & \num{4.83} & \num{0.50} \\
Shapley & \num{2}    & \num{2}    & \num{2}    & \num{2}
        & \num{2}    & \num{2}    & \num{2}
        & \num{2}    & \num{2}    & \num{2}    \\
\rowcolor{gray!10}
Network & \num{0.83}  & \num{1.83}  & \num{-0.67} & \num{-1.50}
        & \num{-0.50} & \num{0.83} & \num{-1.17}
        & \num{1.83}  & \num{2.83}  & \num{-1.50} \\
\bottomrule
\end{tabular}

\end{table}


\begin{table}[h!]
\centering
\setlength{\tabcolsep}{4pt}      
\renewcommand{\arraystretch}{1.1} 

\caption{Interaction values in $(N,v_h)$ for single-player coalitions.}
\label{tab:horse-1}
\begin{tabular}{l*{5}{S[table-format=+2.2]}}
\toprule
 & \multicolumn{5}{c}{Coalition} \\
\cmidrule(lr){2-6}
 & {$\{1\}$} & {$\{2\}$} & {$\{3\}$} & {$\{4\}$} & {$\{5\}$} \\
\midrule
\rowcolor{gray!10}
Myerson & \num{48.33} & \num{7.5} & \num{18.33} & \num{15} & \num{10.83} \\
Shapley & \num{65}    & \num{0}    & \num{0}    & \num{15}    & \num{20}    \\
\rowcolor{gray!10}
Network & \num{-16.67}& \num{7.5}& \num{18.33} & \num{0}& \num{-9.62}\\
\bottomrule
\end{tabular}
\end{table}

\begin{table}
\centering
\setlength{\tabcolsep}{3pt}
\renewcommand{\arraystretch}{1.1}

\caption{Interaction values in $(N,v_h)$ for two-player coalitions }
\label{tab:horse-2}

\begin{tabular}{l*{10}{S[table-format=+2.2]}}
\toprule
 & \multicolumn{10}{c}{Coalition} \\
\cmidrule(lr){2-11}
 & \multicolumn{1}{c}{$\{1,2\}$}
 & \multicolumn{1}{c}{$\{1,3\}$}
 & \multicolumn{1}{c}{$\{1,4\}$}
 & \multicolumn{1}{c}{$\{1,5\}$}
 & \multicolumn{1}{c}{$\{2,3\}$}
 & \multicolumn{1}{c}{$\{2,4\}$}
 & \multicolumn{1}{c}{$\{2,5\}$}
 & \multicolumn{1}{c}{$\{3,4\}$}
 & \multicolumn{1}{c}{$\{3,5\}$}
 & \multicolumn{1}{c}{$\{4,5\}$} \\
\midrule
\rowcolor{gray!10}
Myerson & \num{15} & \num{35} & \num{30} 
        & \num{20} & \num{-30} & \num{15}
        & \num{0} & \num{-15} & \num{20} & \num{-30} \\
Shapley & \num{0}    & \num{0}    & \num{45}    & \num{55}
        & \num{0}    & \num{0}    & \num{0}
        & \num{0}    & \num{0}    & \num{-45}    \\
\rowcolor{gray!10}
Network & \num{15}  & \num{35}  & \num{-15} & \num{-35} & \num{-30} 
        & \num{15} & \num{0} & \num{-15}
        & \num{20}  & \num{15}  \\
\bottomrule
\end{tabular}

\end{table}

The network-induced interaction shows how the same communication structure affects cooperative possibilities differently depending on the underlying game.

In the messages game, see Tables \ref{tab:mess-1} and \ref{tab:mess-2}, symmetry implies that coalitions of the same size have identical interaction values in the absence of communication constraints. Hence, any deviation from this benchmark is entirely driven by the network. For single-player coalitions, the network generates marked asymmetries: centrally positioned players exhibit positive network-induced interaction, whereas peripheral players are penalized. This reproduces the game-theoretical centrality of \cite{GoGoMaOwPoTe03}.  The network also reshapes interactions among two-player coalitions. Those coalitions that are  connected display positive network-induced interaction, while distant pairs experience a reduction. Coalitions involving player 3, who occupies a central position, benefit the most. Thus, the network redistributes complementarities not only across players but also across coalitions, according to their position in the graph.

The effects are even more pronounced in the horse market (see Tables \ref{tab:horse-1} and \ref{tab:horse-2}), where the underlying game already exhibits strong asymmetries. Without communication constraints, the Shapley value allocates all surplus to the seller and the buyers, while intermediaries receive zero payoff. Once the communication graph is taken into account, the pattern of interaction changes substantially. At the individual level, intermediaries obtain positive network-induced interaction despite being null players in the underlying exchange game. Their position allows them to extract bargaining power by controlling access between the seller and the buyers. Conversely, the seller experiences a significant negative network-induced interaction, as the communication structure restricts direct access to buyers. 
At the coalition level, the sign and magnitude of network-induced interaction depend on the structural role of the coalition. Coalitions involving intermediaries tend to gain interaction, whereas coalitions that bypass them or compete for access to the seller tend to lose it. Moreover, the different access of buyers 4 and 5 to intermediaries 2 and 3 endogenously creates or suppresses complementarities between pairs of players, independently of the underlying exchange values.

Together, these examples show that network-induced interaction captures both the redistribution of bargaining power across players and the reorganization of complementarities across coalitions generated by the communication structure.

\section{Conclusions}
\label{sec:conclusions}

This paper extends the Shapley Interaction Index to cooperative environments in which
cooperation is constrained by an underlying communication structure. Starting from a
communication situation $(N,v,\Gamma)$, we follow Myerson's approach and take the
graph-restricted game $v_{\Gamma}$ as the appropriate representation of what coalitions
can actually attain under network constraints. Building on this idea, we propose the
\emph{Myerson interaction index} as the Shapley Interaction Index of the restricted game,
$MI^{v,\Gamma}(S)=SI^{v_{\Gamma}}(S)$, thereby capturing in a single object both the intrinsic
complementarities encoded in $v$ and the relational constraints imposed by $\Gamma$.

A main contribution of the paper is the proposal of an axiomatic system that extends the characterizations of both the Shapley Interaction Index and the Myerson value. As a consequence, first-order Myerson interactions coincide with the Myerson value, while for complete graphs the Myerson interaction index coincides with the Shapley interaction index. In particular, we extend the fairness principle from allocations to higher-order interaction effects, and we also adapt the notion of component efficiency. First, we guarantee that first-order interactions coincide with the Myerson value within each connected component. We then introduce an appropriate notion of a \emph{graph null player}. Together with the extended fairness principle, linearity, and a strong reduced veto partnership consistency condition, these axioms lead to a uniqueness result: the Myerson interaction index is the \emph{only} graph interaction index satisfying the proposed system of requirements. Moreover, we establish the logical independence of the axioms, thereby clarifying the precise role played by each normative axiom in the characterization.

Beyond the axiomatic characterization, the proposed framework provides an interpretable decomposition of how the network reshapes cooperation. By comparing interaction in the original game and in the restricted game, we define \emph{network-induced interaction} as
$$
NI^{v,\Gamma}(S) \coloneq MI^{v,\Gamma}(S) - SI^{v}(S),
$$
which isolates the effect of the communication structure from the purely game-theoretic
effect of $v$. At the individual level, $NI^{v,\Gamma}(i)$ coincides with the
game-theoretical centrality of \cite{GoGoMaOwPoTe03} when symmetric games are considered, while at higher orders it provides a
coalition-level notion of network-driven synergy or redundancy.

The examples highlight that the same communication structure can reorganize
interaction patterns in fundamentally different ways depending on the underlying game.
In a fully symmetric setting (the messages game), all deviations from the uniform
benchmark are entirely attributable to the network, generating asymmetries across both
players and coalitions according to their position in the graph. In the horse market,
where the underlying game is already asymmetric, the communication structure can
endogenously create bargaining power for intermediaries and reduce the interaction of
players who are valuable in the unrestricted market but lack direct access. More
generally, these examples illustrate that network constraints do not merely scale down
payoffs: they redistribute interaction, reshape complementarities across coalitions, and
may generate effective power for agents who are otherwise null in the underlying game.

Overall, the Myerson interaction index provides a coherent and tractable extension of Shapley-based interaction analysis to restricted communication environments, preserving the normative
appeal of the Shapley value. This opens the door to applications in which the relevant object is not only the topology of $\Gamma$, but also the  value $v$ generated by groups of nodes.
In particular, the framework supports \emph{counterfactual} analyses of network evolution, by quantifying how the establishment or removal of specific links (or sets of links)
reorganizes complementarities and redistributes interaction and bargaining power across players and coalitions.

Several directions remain for future research. On the methodological side, further work may develop computational methods and approximations for large player sets, and investigate alternative families of
graph interaction indices (e.g., Banzhaf-type counterparts) under analogous axiomatic systems. Finally, empirical applications may use network-induced interaction to quantify how relational environments reallocate cooperative gains and to identify coalitions whose
synergies are primarily structural rather than intrinsic to the underlying game.


\bibliographystyle{apalike}
\bibliography{references}

@article{FlMoTe19,
  title={Evaluating groups with the generalized {S}hapley value},
  author={Flores, R. and Molina, E. and Tejada, J.},
  journal={4OR},
  volume={17},
  number={2},
  pages={141--172},
  year={2019}
}

@article{FuKoMa06,
  title={Axiomatic characterizations of probabilistic and cardinal-probabilistic interaction indices},
  author={Fujimoto, K. and Kojadinovic, I. and Marichal, J.L.},
  journal={Games and Economic Behavior},
  volume={55},
  number={1},
  pages={72--99},
  year={2006}
}

@article{GoGoMaOwPoTe03,
 title={Centrality and power in social networks: a game theoretic approach},
  author={G\'{o}mez, D. and Gonz\'{a}lez-Arang\"{u}ena, E. and Manuel, C. and Owen, G. and del Pozo, M. and Tejada, J.},
  journal={Mathematical Social Sciences},
  volume={46},
  number={1},
  pages={27--54},
  year={2003}
}

@article{Gr96,
    author = {Grabisch, M.},
    title = {The application of fuzzy integrals in
multicriteria decision making} ,
    journal = {European Journal of Operational Research},
    volume={89},
    pages={445--456},
    year={1996}
}

@article{Gr97a,
  title={\bd{a} {K}-order additive discrete fuzzy measures and their representation},
  author={Grabisch, M.},
  journal={Fuzzy Sets and Systems},
  volume={92},
  number={2},
  pages={167--189},
  year={1997}
}

@incollection{GrLa2016,
  title={Fuzzy measures and integrals in MCDA},
  author={Grabisch, M. and Labreuche, C.},
  booktitle = {Multiple criteria decision analysis: state of the art surveys},
  publisher = {NY: Springer New York},
  year={2016},
  pages = {553--603},
  address = {New York}
}

@article{GrRo99,
  title={An axiomatic approach to the concept of interaction among players in cooperative games},
  author={Grabisch, M. and Roubens, M.},
  journal={International Journal of Game Theory},
  volume={28},
  number={4},
  pages={547--565},
  year={1999}
  }

@incollection{Ha59,
  author = {Harsanyi, J.C.},
  title = {A bargaining model for the cooperative n-person game},
  booktitle = {Contributions to the Theory of Games (AM-40), Vol. 4},
  publisher = {Princeton University Press},
  year = {1959},
  editor = {Tucker, A.W. and Luce, D.R.},
  pages = {325--356},
  address = {Princeton}
}

@article{KaSa87,
  title={On weighted {S}hapley values},
  author={Kalai, E. and Samet, D.},
  journal={International Journal of Game Theory},
  volume={16},
  number={3},
  pages={205--222},
  year={1987}
}

@article{My77,
  title={Graphs and cooperation in games},
  author={Myerson, R.B.},
  journal={Mathematics of Operations Research},
  volume={2},
  number={3},
  pages={225--229},
  year={1977}
}

@article{Ow72,
  title={Multilinear extensions of games},
  author={Owen, G.},
  journal={Management Science},
  volume={18},
  number={5-part-2},
  pages={64--79},
  year={1972},
  publisher={INFORMS}
}

@article{Ow86,
  title={Values of graph-restricted games},
  author={Owen, G.},
  journal={SIAM Journal on Algebraic Discrete Methods},
  volume={7},
  number={2},
  pages={210--220},
  year={1986}
}

@incollection{CaInML2025,
author = {P\'erez-Sechi, C. I. and  Guti\'errez I. and Castro J. and G\'omez   D. and  Mart\'in, D. and Esp\'inola, R.},
booktitle = {2025 IEEE International Conference on Fuzzy Systems (FUZZ)},
 publisher = {IEEE},
 year = {2025},
 pages = {1--6}
 }

@article{SzBaMiRa2015,
  title={The Game-Theoretic Interaction Index on Social Networks with Applications to Link Prediction and Community Detection},
  author={Szczepanski, P. L. and Barcz, A. S. and Michalak, T. P. and Rahwan, T.},
  journal={IjCAI},
  volume={15},
  pages={638--644},
  year={2015}
}

@incollection{TaMiWool2019,
  author = {Tarkowski, M. and Michalak, T. and Wooldridge, M.},
  title = {A game-theoretic algorithm for link prediction},
  booktitle = {arXiv:1912.12846},
  year = {2019},
  publisher = {arXiv preprint}
}

@incollection{Sh53,
  author = {Shapley, L.S.},
  title = {A value for n-person games},
  booktitle = {Contributions to the Theory of Games, Vol. 2},
  publisher = {Princeton University Press},
  year = {1953},
  editor = {Kuhn, H.W. and Tucker, A.W.},
  pages = {307--317},
  address = {Princeton}
}

@book{SlNo12,
  author = {Slikker, M. and van den Nouweland, A.},
  title = {Social and economic networks in cooperative game theory},
  volume = {27},
  publisher = {Springer Science+Business Media, LLC},
  year = {2012}
}

@incollection{SuDhAg20,
  author = {Sundararajan, M. and Dhamdhere, K. and Agarwal, A.},
  title = {The shapley taylor interaction index.},
  booktitle = {International conference on machine learning},
  year = {2020},
  pages = {9259–9268},
  publisher = {PMLR}
}

@article{TsYeRa23,
  title={Faith-shap: The faithful shapley interaction index},
  author={Tsai, C.P. and Yeh, C.K. and Ravikumar, P.},
  journal={Journal of Machine Learning Research},
  volume={24},
  number={94},
  pages={1--42},
  year={2023}
}

\appendix

\section{Logical independence of the axioms}
\label{LIA}

In what follows, we show that each property considered in Theorem~\ref{th:MII_charact} is necessary to guarantee the uniqueness of the Myerson Interaction Index.

\vspace*{1em}
\noindent \textbf{$I$-Component Efficiency}.
Let us consider the Banzhaf interaction index \citep{GrRo99}, defined by
\begin{equation*}
  BI^v(S) \coloneq \sum_{T \subseteq N \setminus S} {\frac{1}{2^{n-s}}\,\delta^v_S(T)}.
\end{equation*}
We may extend it to the Banzhaf graph interaction index by $BI^{v,\G}(S) \coloneq BI^{v^\G}(S)$ for all $\G \in \cG^N$ and all $S \subseteq N$. This index is well defined, i.e. it satisfies conditions~\eqref{eq:inter_comp_indep}, since, using the Harsanyi-dividend representation of the Banzhaf interaction index
\begin{equation*}
  BI^{v^\G}(S) = \sum_{T \supseteq S} {\frac{1}{2^{t-s}}\,\Delta_{v^\G}(T)}
\end{equation*}
which may be found in \cite{GrRo99}, the same reasoning in the proof of Proposition~\ref{prop:MII_gii} applies.

Moreover, by arguments analogous to those in Proposition~\ref{prop:MII_propert}, the Banzhaf graph interaction index satisfies $I$-GN, $I$-F, $I$-L, and $I$-SRVPC. However, since Banzhaf value is not efficient, but rather 2-efficient \citep{GrRo99}, the index does not satisfy $I$-CE. It suffices to choose a game $v \in G^N$ such that $\sum_{i \in N } {BI^{v}(i)} \neq v(N)$ and $\G$ the complete graph on $N$ so that $BI^{v}(i) = BI^{v,\G}(i)$ for all $i \in N$.

\vspace*{1em}
\noindent \textbf{$I$-Graph Null}. 
Define a graph interaction index $GI$ as follows. For every $N \in \cN$ and $\G \in \mathcal G^N$, let
\begin{equation*}
  GI^{v,\G}(S) \coloneq 0, \quad\forall S \subseteq N
\end{equation*}
for every null game $(N,v) \in G^N$; and
\begin{equation*}
  GI^{v,\G}(S) \coloneq \sum_{T \in 2^N \setminus\{\emptyset\}} {\Delta_v(T)\,GI^{u_T,\G}(S)},\quad\forall S \subseteq N
\end{equation*}
for every non-null game $(N,v) \in G^N$. Thus, $GI$ is defined on the basis of unanimity games $\{u_T\,|\,T \in 2^N \setminus\{\emptyset\}\}$. Specifically, given $T \in 2^N \setminus\{\emptyset\}$, let
\begin{equation*}
  GI^{u_T,\G}(S) \coloneq MI^{u_T,\G}(S) + f_{\alpha}^{T,\G}(S),\quad\forall S \subseteq N
\end{equation*}
where the function $f_{\alpha}^{T,\G}: 2^N \setminus \{\emptyset\} \rightarrow \SR$ is defined for $\alpha > 0$ by
\begin{equation*}
  f_{\alpha}^{T,\G}(S) \coloneq
  \begin{cases}
    \alpha, & \text{if $|S| \geq 2$ and $S$ satisfies conditions (i) to (iii)} \\
         0, & \text{otherwise}
  \end{cases}
\end{equation*}
with conditions:
\begin{enumerate}[label=(\roman*)]
  \item $T \subseteq S$.
  \item $S$ is connected in $\G$.
  \item $\nexists E' \subseteq w(S) \coloneq \{\{i,j\}\in E\,|\, i\in S,\ j\notin S\}$ such that $S$ becomes a veto graph partnership in $\G^{-E'}$.
\end{enumerate}
Note that, with $E' = \emptyset$ in (iii), we require in particular that $S$ is not a veto graph partnership with respect to $\G$.

Clearly, $GI$ is well defined in terms of conditions~\eqref{eq:inter_comp_indep} as it essentially coincides with the Myerson interaction index in unanimity games $u_T$ except for $S \subseteq N$ with $|S| \geq 2$ satisfying (i) to (iii). In such cases, due to (i) and (ii), $T \subseteq S \subseteq C$ with $C \in \cK_\G$ and $f_{\alpha}^{T,\G}(S) = \alpha = f_{\alpha}^{T,\G_C}(S)$ follows with $u_T\vert_C$ a non-null game, so $GI^{u_T,\G}(S) = GI^{u_T\vert_C,\G_c}(S)$ holds. In addition, $I$-L is clearly satisfied by construction, as well as $I$-CE because $GI$ coincides with $MI$ for first-order interactions.

To verify $I$-F and $I$-SRVPC, by construction it suffices to do it for unanimity games. Given $T \in 2^N \setminus \{\emptyset\}$ and $\G \in \cG^N$, we begin with $I$-F. Consider $\{i,j\} \in E$ and $S \subseteq N \setminus \{i,j\}$. If $S \cup i$ verifies conditions (i) to (iii) in $\G$, then it keeps verifying all conditions with respect to $\G_{-\{i,j\}}$ as removing edge $\{i,j\}$ with $i \in S \cup i$ and $j \notin S \cup i$ does not affect them. Similarly, if any condition is not met by $S \cup i$ in $\G$, removing edge $\{i,j\}$ cannot result in them being fulfilled. Therefore, $f_{\alpha}^{T,\G}(S \cup i) = f_{\alpha}^{T,\G_{-\{i,j\}}}(S \cup i)$ and it holds
\begin{equation*}
  GI^{u_T,\G}(S \cup i) - GI^{u_T,\G_{-\{i,j\}}}(S \cup i) = MI^{u_T,\G}(S \cup i) - MI^{u_T,\G_{-\{i,j\}}}(S \cup i).
\end{equation*}
The same argument applies to $S \cup j$ and, as $MI$ verifies $I$-F, consequently so does $GI$.

Secondly, to show $I$-SRVPC, consider $P \subseteq N$ a veto graph partnership in $\G$ for $u_T$, i.e.~$P \subseteq T \cup EB_{\G}(T)$ with $P \neq \emptyset$, and $S \subseteq N \setminus P$. If $S = \emptyset$, then $f_{\alpha}^{T,\G}(P \cup S) = 0$ due to condition (iii). In addition, $|[P] \cup S| = 1$, so $f_{\alpha}^{T_{[P]},\G_{[P]}}([P] \cup S) = 0$ as well. Hence, it holds
\begin{equation*}
  GI^{u_T,\G}(P \cup S) = MI^{u_T,\G}(P \cup S) =
  MI^{u_{T,[P]},\G_{[P]}}([P] \cup S) = GI^{u_{T,[P]},\G_{[P]}}([P] \cup S)
\end{equation*}
as $MI$ satisfies $I$-SRVPC. Otherwise, $|P \cup S| \geq 2$ and it follows that $P \cup S$ verifies conditions (i) to (iii) if and only if $[P] \cup S$ does (with respect to the corresponding quotient game). As a result, $f_{\alpha}^{T,\G}(P \cup S) = f_{\alpha}^{T_{[P]},\G_{[P]}}([P] \cup S)$ so, in a similar manner as before, $I$-SRVPC holds.

However, $GI$ violates $I$-GN since, for every non-empty $T \subsetneq N$ and connected graph $\G \in \cG^N$ such that $D_0^{\G}(u_T) = N \setminus (T \cup B_{\G}(T)) \neq \emptyset$, we have
\begin{equation*}
  GI^{u_T,\G}(N) = MI^{u_T,\G}(N) + \alpha = \alpha > 0
\end{equation*}
as $N$ satisfies conditions (i) to (iii).

\vspace*{5mm}
\noindent \textbf{$I$-Fairness}. 
Define a graph interaction index $GI$ on the basis of unanimity games $\{u_T\,|\,T \in 2^N \setminus\{\emptyset\}\}$ as in the previous discussion, but with the following specification for $GI^{u_T,\G}$ for every $N \in \cN$, $\G \in \cG^N$ and $T \in 2^N \setminus\{\emptyset\}$:
\begin{equation*}
  GI^{u_T,\G}(S) \coloneq
  \begin{cases}
    |S_{-\cE(T)}|\,MI^{u_T,\G}(S),& \text{if } S_{\cE(T)} \neq \emptyset \text{ and } S_{-\cE(T)} \neq \emptyset \\
    MI^{u_T,\G}(S),& \text{otherwise}
\end{cases}
\end{equation*}
where $S_{\cE(T)} \coloneq S \cap (T \cup EB_\G(T))$ and $S_{-\cE(T)} \coloneq S \cap (B_\G(T) \setminus EB_\G(T))$. Note that, if $S = i$, then $S_{\cE(T)} = \emptyset$ or $S_{-\cE(T)} = \emptyset$ and thus $GI^{u_T,\G}(i) = MI^{u_T,\G}(i)$.

First, we check that $GI$ satisfies conditions~\eqref{eq:inter_comp_indep} and thus is well defined. Let $S \subseteq N$ be such that $\nexists C \in \cK_{\G}$ with $S \subseteq C$. Hence, being $MI$ a graph interaction index, $MI^{u_T,\G}(S) = 0$ and $GI^{u_T,\G}(S) = 0$ follows. On the other hand, let $S \subseteq N$ be such that $S \subseteq C$ with $C \in \cK_{\G}$. If $T$ cannot be connected in $\G$, then
\begin{equation*}
  GI^{u_T,\G}(S) = MI^{u_T,\G}(S) = 0
\end{equation*}
with the last equality a consequence of $u_T$ restricted to $\G$ being a null game. Otherwise, $T$ pertains to a connected component in $\G$. If this connected component is not $C$, then $S_{\cE(T)} = S_{-\cE(T)} = \emptyset$ with respect to both $\G$ and $\G_C$ and it follows
\begin{equation*}
  GI^{u_T,\G}(S) = MI^{u_T,\G}(S) = MI^{u_T\vert_C,\G_c}(S) = GI^{u_T\vert_C,\G_c}(S)
\end{equation*}
because $MI$ verifies conditions \eqref{eq:inter_comp_indep}. If $T \subseteq C$, then the identification of $S_{\cE(T)}$ and $S_{-\cE(T)}$ coincides in $\G$ and the restricted graph $\G_C$ so, in the same manner as before, $MI^{u_T,\G}(S) = MI^{u_T\vert_C,\G_c}(S)$ and consequently $GI^{u_T,\G}(S) = GI^{u_T\vert_C,\G_c}(S)$.

As in the previous discussion, $I$-L is clearly satisfied by construction, as well as $I$-CE because $GI$ coincides with $MI$ for first-order interactions. $I$-GN is also immediate since for every unanimity game $u_T$, graph $\G$ and graph null player $i \in D_0^{\G}(u_T)$ the Myerson interaction index verifies $MI^{u_T,\G}(S \cup i) = 0$ for all $S \subseteq N \setminus i$ and $GI^{u_T,\G}(S \cup i) = 0$ follows by definition.

Concerning $I$-SRVPC, it suffices to check it for unanimity games. If $T$ cannot be connected in $\G$, then $u_T$ restricted to $\G$ is a null game and
\begin{equation*}
  GI^{u_T,\G}(P \cup S) = 0 = GI^{u_{T,[P]}^{\G},\G_{[P]}}([P] \cup S),\quad\forall S \subseteq N \setminus P
\end{equation*}
for every veto graph partnership $P$. Otherwise, let $P \subseteq T \cup EB_\G(T)$ be a veto graph partnership in $(N,u_T,\G)$ and $S \subseteq N \setminus P$. Since $\emptyset \neq P \subseteq (P \cup S)_{\cE(T)}$, two cases may arise:
\begin{enumerate}[label=(\roman*)]
  \item If $(P \cup S)_{-\cE(T)} = \emptyset$, then
  \begin{equation*}
    GI^{u_T,\G}(P \cup S) = MI^{u_T,\G}(P \cup S) =
    MI^{u_{T,[P]}^{\G},\G_{[P]}}([P] \cup S) = GI^{u_{T,[P]}^{\G},\G_{[P]}}([P] \cup S)
  \end{equation*}
  as $MI$ satisfies $I$-SRVPC.
  \item If $(P \cup S)_{-\cE(T)} \neq \emptyset$, in a similar way as before, it follows
  \begin{equation*}
    \begin{array}{r l}
      GI^{u_T,\G}(P \cup S) & = |S_{-\cE(T)}|\,MI^{u_T,\G}(P \cup S) \\
      & = |S_{-\cE(T)}|\,MI^{u_{T,[P]}^{\G},\G_{[P]}}([P] \cup S) = GI^{u_{T,[P]}^{\G},\G_{[P]}}([P] \cup S),
    \end{array}
  \end{equation*}
  where the last equality holds since $|S_{-\cE(T)}|$ is invariant under the quotient operation with respect to any veto graph partnership $P$ of $(N,u_T,\G)$.
\end{enumerate}

However, $GI$ does not satisfy $I$-F, as shown by the following counterexample. Let $(N,v) = (\{1,2,3,4,5\},u_{\{1,5\}})$ and $\G = (N,E)$, with
\begin{equation*}
  E=\{\{1,2\},\{1,3\},\{1,4\},\{2,5\},\{3,5\},\{4,5\}\}. 
\end{equation*}
Then $B_\G(\{1,5\}) = \{2,3,4\}$ and $EB_\G(\{1,5\}) = \emptyset$. Consider the edge $\{2,5\}$ and the set $S = \{3,4\} \subseteq N\setminus \{2,5\}$. The following calculations hold:
\begin{itemize}
  \item $\displaystyle GI^{v,\G}(\{2,3,4\}) = MI^{v,\G}(\{2,3,4\}) = \sum_{L \supseteq \{2,3,4\}} {\frac{\Delta_{v^{\G}}(L)}{l-2}} = \frac{1}{3}$
  \item $GI^{v,\G_{-\{2,5\}}}(\{2,3,4\}) = 0$ being $2$ a graph null player
  \item $\displaystyle GI^{v,\G}(\{3,4,5\}) = 2\,MI^{v,\G}(\{3,4,5\}) = 2 \sum_{L \supseteq \{3,4,5\}} {\frac{\Delta_{v^{\G}}(L)}{l-2}}  = -\frac{1}{3}$
  \item $\displaystyle GI^{v,\G_{-\{2,5\}}}(\{3,4,5\}) = 2\,MI^{v,\G_{-\{2,5\}}}(\{3,4,5\}) = 2 \sum_{L \supseteq \{3,4,5\}} {\frac{\Delta_{v^{\G_{-\{2,5\}}}}(L)}{l-2}} = -1$
\end{itemize}
Thus, $I$-F fails, since
\begin{equation*}
  GI^{v,\G}(\{2,3,4\}) - GI^{v,\G_{-\{2,5\}}}(\{2,3,4\}) = \frac{1}{3} \neq
  \frac{2}{3} = GI^{v,\G}(\{3,4,5\}) - GI^{v,\G_{-\{2,5\}}}(\{3,4,5\}).
\end{equation*}

\vspace*{5mm}
\noindent \textbf{$I$-Strong Reduced Veto Partnership Consistency}. Define a graph interaction index $GI$ as follows. For every $N \in \cN$ and $\G \in \cG^N$, set $GI^{v,\G}(i) \coloneq \mu_i^v(\G)$, for all singletons $\{i\} \subseteq N$, and $GI^{v,\G}(S) \coloneq 0$, for all $S \subseteq N$ with $|S| \geq 2$.

In order to verify that $GI$ is well defined we check conditions~\eqref{eq:inter_comp_indep}. Let $(N,v,\G)$ be a communication situation. Whenever $S \subseteq N$ is not a singleton, then $GI^{v,\G} = GI^{v\vert_C,\G_C}(S) = 0$ and the conditions hold. If $S$ is a singleton, then $GI$ coincides with the Myerson interaction index and conditions are also fulfilled.

Moreover, $GI$ satisfies $I$-CE, since first-order interactions coincide with the Myerson value; $I$-GN is trivially verified because all involved interactions take a value of $0$; and, finally, $I$-F and $I$-L are also trivial in the same manner except for singletons, which derive from the Myerson value properties of \emph{fairness} and \emph{linearity}. 

However, $GI$ does not satisfy $I$-SRVPC, as shown by the following counterexample. Let $(N,v) = (\{1,2\},u_{\{1,2\}})$ and $\G = (N,E)$, with $E = \{\{1,2\}\}$. Clearly, $P = \{1,2\}$ is a veto graph partnership, but
\begin{equation*}
  GI^{v,\G}(\{1,2\}) = 0 \neq 1 = GI^{v_{[1,2]}^{\G},\G_{[1,2]}}([1,2])
\end{equation*}
where the last interaction takes the Myerson value of a single player in a non-null unanimity game.

\vspace*{5mm}
\noindent \textbf{$I$-Linearity}. Define a graph interaction index $GI$ for all communication situations $(N,v,\G)$ by
\begin{equation*}
  GI^{v,\G}(S) \coloneq
  \begin{cases}
    MI^{v^2,\G}(S),& \text{if } |S| \geq 2 \text{ and } v(T) > 0, \text{ for all } T \neq \emptyset\\
    MI^{v,\G}(S),& \text{otherwise}
  \end{cases}
\end{equation*}
where $v^2$ refers to the game where the payoffs of every coalition are the squared value of that of game $v$.

The $GI$ graph interaction index is well defined, since the Myerson interaction index satisfies conditions~\eqref{eq:inter_comp_indep} for every game in $G^N$. It verifies $I$-CE, since first-order interactions coincide with the Myerson value. Concerning $I$-GN, note that a player is graph null in $(N,v,\G)$ if and only if it is a graph null player in $(N,v^2,\G)$. Thus, $GI$ satisfies $I$-GN as $MI$ does for every game in $G^N$.

Furthermore, note that if $v$ is a strictly positive game, then removing edges from the communication graph $\G$ does not generate zero-value coalitions. As a consequence, $GI$ verifies $I$-F because the Myerson interaction index does for every game in $G^N$ and there is no equation involving $MI$ with respect to both $v$ and $v^2$.

Finally, with respect to $I$-SRVPC, let $P$ be a veto graph partnership in communication situation $(N,v,\G)$. Without loss of generality assume $|P| \geq 2$. If $P \subsetneq N$, then let $i \in N \setminus P$. Being $P$ a veto graph partnership, $v(i) = v_{[P]}^{\G}(i) = 0$ and both $v$ and its associated quotient game take zero-values. Therefore, $GI$ coincides with $MI$ and satisfies $I$-SRVPC. On the other hand, if $P = N$, then $GI^{v,\G}(P) = MI^{v,\G}(P)$ as $|P| \geq 2$ and $v(i) = 0$, for all $i \in P$. In addition, $GI^{v_{[P]}^{\G},\G_{[P]}}([P]) = MI^{v_{[P]}^{\G},\G_{[P]}}([P])$ because $|[P]| < 2$. As a consequence, $GI$ also coincides with $MI$ and satisfies $I$-SRVPC.

However, by construction, $GI$ does not satisfy $I$-L. It suffices to consider two strictly positive games $v_1$ and $v_2$ with at least two players and a complete graph $\G$ to show that
\begin{equation*}
  GI^{v_1+v_2,\G}(N) = MI^{(v_1+v_2)^2,\G}(N) \neq
  MI^{v_1^2,\G}(S) + MI^{v_2^2,\G}(S) = GI^{v_1,\G}(S) + GI^{v_2,\G}(S)
\end{equation*}
because $MI$ is linear and $(v_1+v_2)^2 = v_1^2+2v_1v_2+v_2^2 \neq v_1^2+v_2^2$ with $v_1,v_2 > 0$.
\end{document}